\documentclass{amsart}

\usepackage[all, cmtip]{xy}

\usepackage{amssymb,amsmath,amscd,xy,graphicx,textcomp,mathtools}
\usepackage{epsf} 
\usepackage{titlepic}
\usepackage{hyperref}
\usepackage{tabulary}
\usepackage{booktabs}

\hypersetup{
  colorlinks   = true, 
  urlcolor     = blue, 
  linkcolor    = blue, 
  citecolor   = blue 
}

\newtheorem{theorem}{Theorem}[section]
\newtheorem{lemma}[theorem]{Lemma}
\newtheorem{corollary}[theorem]{Corollary}
\newtheorem{definition}[theorem]{Definition}

\newtheorem{remark}[theorem]{\it Remark}
\newtheorem{example}[theorem]{Example}
\newtheorem{proposition}[theorem]{Proposition}
\newtheorem{conjecture}[theorem]{Conjecture}

\xyoption{arrow}

\xyoption{matrix}

\def\C{\mathbb{C}}
\def\R{\mathbb{R}}
\def\Z{\mathbb{Z}}

\def\tree{\mathcal{T}}
\def\Z{\mathbb{Z}}

\title{The Algebra of Conformal Blocks}
\author{Christopher Manon}
\thanks{This work was supported by the NSF fellowship DMS-0902710}

\begin{document}

\begin{abstract}
For each simply  connected, simple complex group $G$ we show that the direct sum of all vector bundles of conformal blocks on the moduli stack $\bar{\mathcal{M}}_{g, n}$ of stable marked curves carries the structure of a flat sheaf of commutative algebras.  The fiber of this sheaf over a smooth marked curve $(C, \vec{p})$ agrees with the Cox ring of the moduli of quasi-parabolic principal $G-$bundles on $(C, \vec{p})$.   We use the factorization rules on conformal blocks to produce flat degenerations of these algebras.   In the $SL_2(\C)$ case, these degenerations result in toric varieties which appear in the theory of phylogenetic statistical varieties, and the study of integrable systems in the moduli of rank $2$ vector bundles. We conclude with a combinatorial proof that the Cox ring of the moduli stack of quasi-parabolic $SL_2(\C)$ principal bundles over a generic curve is generated by conformal blocks of levels $1$ and $2$ with relations generated in degrees $2, 3,$ and $4$.   
\end{abstract}

\maketitle

\tableofcontents

\smallskip
\noindent



\section{Introduction}

Let $(C, \vec{p})$ be a smooth, complex, projective curve with distinct marked points $\vec{p} = \{p_1, \ldots, p_n\}$  $\subset C$, and let $G$ be a simple, simply connected complex group.  We fix a Borel subgroup $B$ and choose a parabolic $\Lambda_i$ containing $B$ for each $p_i$.   A quasi-parabolic principal $G-$bundle on $C$ of type $\vec{\Lambda} = \{\Lambda_1, \ldots, \Lambda_n\}$ is a principal $G-$bundle $E \to C$ together with a choice of a point $\rho_i$ in the fiber over $p_i$ of the associated bundle $E\times_G (G/\Lambda_i)$ (equivalently, a choice of right $\Lambda_i$ orbit in the fiber of $E$ over $p_i$).  We study the Cox ring, or total coordinate ring, of the moduli stack $\mathcal{M}_{C, \vec{p}}(\vec{\Lambda})$ of these objects.

The Cox ring associated to this moduli problem is the direct sum of all the spaces of global sections of line bundles on the stack taken over the torsion free part of the Picard group.  The group $Pic(\mathcal{M}_{C, \vec{p}}(\vec{\Lambda}))$ is computed by Laszlo and Sorger in \cite{LS}, where they show it is a free Abelian group:

\begin{equation}\label{pic}
Pic(\mathcal{M}_{C, \vec{p}}(\vec{\Lambda})) = \mathcal{X}(\Lambda_1)\times \ldots \times \mathcal{X}(\Lambda_n) \times \Z.\\
\end{equation}

\noindent
Here $\mathcal{X}(\Lambda_i)$ is the group of characters of the parabolic subgroup $\Lambda_i \subset G.$ A celebrated result of  Faltings \cite{F}, Kumar, Narasimhan, Ramanathan \cite{KNR}, Beauville, Laszlo, Sorger \cite{BL}, \cite{BLS}, \cite{S}, and Pauly \cite{P}  identifies global sections of line bundles on $\mathcal{M}_{C, \vec{p}}(\vec{\Lambda})$ with the spaces of conformal blocks from the Wess-Zumino-Novikov-Witten (WZNW) model of conformal theory.
For a fixed curve $(C, \vec{p})$, there is one such space $ \mathcal{V}^{\dagger}_{C, \vec{p}}(\vec{\lambda}, L)$ for each choice of dominant weights $\lambda_i \in \mathcal{X}(\Lambda_i)$ and a non-negative integer $L \in \Z_{\geq 0}:$ 

\begin{equation}\label{lineblock}
H^0(\mathcal{M}_{C, \vec{p}}(\vec{\Lambda}), \mathcal{L}(\vec{\lambda}, L)) = \mathcal{V}^{\dagger}_{C, \vec{p}}(\vec{\lambda}, L).\\
\end{equation}

The Cox ring of $\mathcal{M}_{C, \vec{p}}(\vec{\Lambda})$ is therefore the sum of all the spaces of conformal blocks with compatible parabolic data:

\begin{equation}
Cox(\mathcal{M}_{C, \vec{p}}(\vec{\Lambda})) = \bigoplus_{\vec{\lambda}, L} H^0(\mathcal{M}_{C, \vec{p}}(\vec{\Lambda}), \mathcal{L}(\vec{\lambda}, L)) = \bigoplus_{\vec{\lambda}, L} \mathcal{V}^{\dagger}_{C, \vec{p}}(\vec{\lambda}, L).\\
\end{equation}

 The main theorem of this paper produces a family of flat degenerations of $Cox(\mathcal{M}_{C, \vec{p}}(\vec{\Lambda}))$  from the combinatorial properties of the WZNW theory.  We state this theorem for $\Lambda_i = B \subset G$ a Borel subgroup, denoted $\mathcal{M}_{C, \vec{p}}(G)$, as all other cases are implied by this case.  In what follows, $\Gamma$ is a graph with non-leaf vertex set $V(\Gamma)$ and edge set $E(\Gamma)$. 

\begin{theorem}\label{main}
For $C, \vec{p}$ a marked stable curve, there is a flat degeneration of $Cox(\mathcal{M}_{C, \vec{p}}(G))$ for every trivalent graph $\Gamma$ with first Betti number $g = genus(C)$ and $n = |\vec{p}|$ leaves:

\begin{equation}
Cox(\mathcal{M}_{C, \vec{p}}(G))  \Rightarrow [\bigotimes_{v \in V(\Gamma)} Cox(\mathcal{M}_{0,3}(G))]^{T_{\Gamma}}.\\
\end{equation}

\noindent
Here $T_{\Gamma}$ is a product of $|E(\Gamma)| - n$ tori $T\times \C^*$, where $T \subset G$ a maximal torus. 
\end{theorem} 

For a description of the action of $T_{\Gamma}$ on $\bigotimes_{v \in V(\Gamma)} Cox(\mathcal{M}_{0,3}(G))$ see Section \ref{filter}. 
Theorem \ref{main} is a ``ringification'' of the $factorization$ $rules$ for conformal blocks (see Section \ref{filter}); and in keeping with the flavor of factorization, many algebraic properties of the Cox ring of $\mathcal{M}_{C, \vec{p}}(G)$ can be understood in terms of the $3-$pointed, genus $0$ case.  For example, in the $SL_2(\C)$ case (discussed in more detail below), Theorem \ref{main} is used to show that $Cox(\mathcal{M}_{C, \vec{p}}(SL_2(\C))$ is generically finitely generated (Theorem \ref{degensl2} and Theorem \ref{sl2gen}).   Furthermore, it can be shown that $Cox(\mathcal{M}_{C, \vec{p}}(SL_2(\C))$ is a Gorenstein algebra using an argument along the lines of \cite[Theorem $7.3$]{M4}.

Let $R_{C, \vec{p}}(\vec{\lambda}, L) = \bigoplus_{N \geq 0} H^0(\mathcal{M}_{C, \vec{p}}(N\vec{\Lambda}), \mathcal{NL}(\vec{\lambda}, L))$ $=\bigoplus_{N \geq 0}  \mathcal{V}^{\dagger}_{C, \vec{p}}(N\vec{\lambda}, NL)$  denote the projective coordinate ring corresponding to the line bundle $\mathcal{L}(\vec{\lambda}, L)$ on $\mathcal{M}_{C, \vec{p}}(\vec{\Lambda})$, note that this is a graded subalgebra of  $Cox(\mathcal{M}_{C, \vec{p}}(\vec{\Lambda}))$.
The corresponding coarse moduli space $Proj(R_{C, \vec{p}}(\vec{\lambda}, L))$ is denoted $M_{C, \vec{p}}(\vec{\lambda}, L)$. The celebrated Verlinde formula (\cite{Verlinde}, \cite{F}, \cite{B}) calculates the dimension of the space of global sections $H^0(\mathcal{M}_{C, \vec{p}}(\vec{\Lambda}), \mathcal{L}(\vec{\lambda}, L)).$   We view Theorem \ref{main} as a first step toward developing a polyhedral rule for computing the Verlinde formula, which should correspond to toric degenerations of the coarse moduli $M_{C, \vec{p}}(\vec{\lambda}, L)$.

 A toric degeneration of an algebra is a flat family of algebras over some base, with a special fiber equal to the semigroup algebra of a normal affine semigroup.   Presentation results, and certain algebraic properties (e.g. Gorenstein, Koszul) can be easier to prove on an algebra with a toric degeneration, as these properties are controllable under flat degeneration and are more readily established by combinatorial means on the special fiber of the degeneration.  Theorem \ref{main} reduces the problem to finding such a degeneration for $R_{C, \vec{p}}(\vec{\lambda}, L)$ or $Cox(\mathcal{M}_{C, \vec{p}}(G))$ to finding a degeneration for $ Cox(\mathcal{M}_{0,3}(G))$ which respects the multigrading by dominant weights.  When $G = SL_2(\C),$ the algebra $Cox(\mathcal{M}_{0,3}(SL_2(\C)))$ is already an affine semigroup algebra, so in this case our degenerations are toric.  The relevant affine semigroup algebras are the graded algebras associated to the following polytopes. 

\begin{definition}
For $\Gamma$ a trivalent graph, let $P_{\Gamma}$ be the polytope of weightings $w: E(\Gamma) \to \R_{\geq 0}$ which satisfy the following properties at each internal vertex $v \in V(\Gamma).$

\begin{enumerate}
\item The sum of the three weights incident on a vertex $w_1(v) + w_2(v) + w_3(v)$ is $\leq 2$\\
\item These three weights satisfy the triangle inequalities, $|w_1(v) - w_3(v)| \leq w_2(v) \leq w_1(v) + w_3(v)$\\
\end{enumerate}

\end{definition}

Also see \cite{Bu} for a description of these polytopes.
We consider the $P_{\Gamma}$ and their Minkowski sums with respect to the lattice $\mathcal{L}_{\Gamma} \subset \R^{E(\Gamma)}$ defined by the condition that all edges are weighted with integers, and $w_1(v) + w_2(v) + w_3(v) \in 2\Z$ for all $v \in V(\Gamma).$  An affine semigroup is obtained from $P_{\Gamma}$ by considering the lattice points in the Minkowski sums $L\circ P_{\Gamma} = \{ u_1 + \ldots + u_L | u_i \in P_{\Gamma}\}.$ The union $S(P_{\Gamma}) = \coprod_{L \geq 0} L\circ P_{\Gamma} \cap \mathcal{L}_{\Gamma}$ is naturally graded, and it is easy to see that if $u_1 \in L\circ P_{\Gamma}$ and $u_2 \in K \circ P_{\Gamma}$, then $u_1 + u_2 \in (L + K)\circ P_{\Gamma},$ where the operation $+$ is sum of integer valued functions on $E(\Gamma).$  Let $\C[P_{\Gamma}]$ be the affine semigroup algebra obtained from the graded affine semigroup $S(P_{\Gamma})$; note that $\C[P_{\Gamma}]$ comes with a distinguished basis in bijection with the graded set of lattice points. We show the following in Section \ref{sl2}. 

\begin{theorem}\label{degensl2}
Let $C, \vec{p}$ be an $n$-marked smooth, projective curve of genus $g$, and $\Gamma$ a trivalent graph with first Betti number $g$ and $n$ leaves.  There is a flat degeneration:

\begin{equation}
Cox(\mathcal{M}_{C, \vec{p}}(SL_2(\C))) \Rightarrow \C[P_{\Gamma}].\\
\end{equation}

\end{theorem}

  Recall that dominant weights of $SL_2(\C)$ are non-negative integers, therefore we may associate a projective coordinate ring of the moduli of $SL_2(\C)$ quasi-parabolic principal bundles $R_{C, \vec{p}}(\vec{r}, L)$ to the data $(\vec{r}, L) \in \Z_{\geq 0}^{n+1}$. As a corollary of Theorem \ref{degensl2} we also obtain explicit toric degenerations of spaces $M_{C, \vec{p}}(\vec{r}, L) = Proj(R_{C, \vec{p}}(\vec{r}, L))$. 

\begin{definition}
let $P_{\Gamma}(\vec{r}, L)$ be the polytope obtained as the fiber over $\vec{r}$ for the map
$L \circ P_{\Gamma} \to \R^n$, computed by forgetting all weights except those on the leaf-edges. 

\end{definition}

Let $\C[P_{\Gamma}(\vec{r}, L)]$ be the affine semigroup defined by the polytope $P_{\Gamma}(\vec{r}, L)$. 
For $G = SL_2(\C)$, $C, \vec{p}$ a marked smooth projective curve, and $\Gamma$ a trivalent graph with compatible information, there is a flat degeneration,

\begin{equation}
R_{C, \vec{p}}(\vec{r}, L)) \Rightarrow \C[P_{\Gamma}(\vec{r}, L)].\\
\end{equation}

Theorem \ref{degensl2} is utilized in \cite{M} and \cite{M3} to prove that the projective coordinate ring of the square of any effective
line bundle on the moduli stack $\mathcal{M}_{C, \vec{p}}(SL_2(\C))$ is generated by its degree $1$ elements, and is a Koszul algebra for generic $C, \vec{p}$. This result follows from the analysis of $P_{\Gamma}(\vec{r}, L)$ for particular well-chosen trivalent graphs.  We follow a similar
strategy here, by studying a particular polytope $P_{\Gamma_{g, n}}$, we prove the following. 

\begin{theorem}\label{sl2gen}
For generic $C, \vec{p}$, the algebra $Cox(\mathcal{M}_{C, \vec{p}}(SL_2(\C)))$ is generated by conformal blocks
of level $L = 1, 2$.  The corresponding ideal of relations is generated in levels $2, 3, 4$. 
\end{theorem}

We prove Theorem \ref{sl2gen} by establishing these properties for the algebra $\C[P_{\Gamma_{g, n}}]$ in Section \ref{sl2}, the theorem then holds for  $Cox(\mathcal{M}_{C, \vec{p}}(SL_2(\C)))$ because of general properties of flat families of algebras. In particular, the degrees of generators and relations needed to present a particular algebra in a flat family bound such degrees generically, for more discussion on this point see \cite[Theorem 1.11]{M3}.

 Theorem \ref{sl2gen} is a simultaneous generalization of a theorem of Castravet and Tevelev \cite{CT}, Sturmfels and Xu, \cite{StXu}, and a theorem of Abe \cite{A}.  Castravet and Tevelev, along with Sturmfels and Xu treat the case $g  = 0$, and show that $L = 1$ conformal blocks generate with quadratic relations by utilizing a theorem of Bauer \cite{Ba} and results of Buczy\'{n}ska and Wi\'{s}niewski, \cite{BW}, discussed below. Abe treats the case $n = 0,$ and shows that $L =1$ conformal blocks generate by combining a proof that $L= 1, 2$ suffice with a result of Beauville \cite{B2} which establishes that the $L=1$ component generates $L=2$. 

Theorem \ref{sl2gen} should be of interest in the emerging field of Newton-Okounkov bodies, \cite{KK}, \cite{LM}, \cite{HK}.   We show in Subsection \ref{NOKbody}, Proposition \ref{NOK}, that the polytope $P_{\Gamma}$ is a Newton-Okounkov body of the scheme $Proj(V_{C, \vec{p}}(SL_2(\C))$, when $(C, \vec{p})$ is a stable curve of type $\Gamma.$

\subsection{Organization and methods}

Theorem \ref{main} is a consequence of the commutative algebra analogues
of classical results from the theory of conformal blocks proved in \cite{TUY}.   In what follows $\bar{\mathcal{M}}_{g, n}$ denotes the Deligne Mumford stack of stable $n$-pointed curves of genus $g$ (see \cite{DM}). 

\begin{enumerate}\label{bundle}\label{factorization}
\item (Flatness) The spaces $\mathcal{V}_{C, \vec{p}}^{\dagger}(\vec{\lambda}, L)$ fit together into a vector bundle $\mathcal{V}^{\dagger}(\vec{\lambda}, L)$ over $\bar{\mathcal{M}}_{g, n}$.\\
\item (Factorization) For $\tilde{C}$ the partial normalization of a stable curve $C$ at a double point $q,$ there is an isomorphism of vector spaces,

$$
\begin{CD}
\mathcal{V}^{\dagger}_{C, \vec{p}}(\vec{\lambda}, L) @<\cong<< \bigoplus_{\alpha \in \Delta_L} \mathcal{V}^{\dagger}_{\tilde{C}, \vec{p}, q_1, q_2}(\vec{\lambda}, \alpha, \alpha^*, L)\\
\end{CD}
$$
here the point $q_1$ is always assigned the dual weight $\alpha^*$ to the weight $\alpha$ assigned to its partner $q_2$ (see Figure \ref{fig:4}). \\
\end{enumerate}

\begin{figure}[htbp]
\centering
\includegraphics[scale = 0.35]{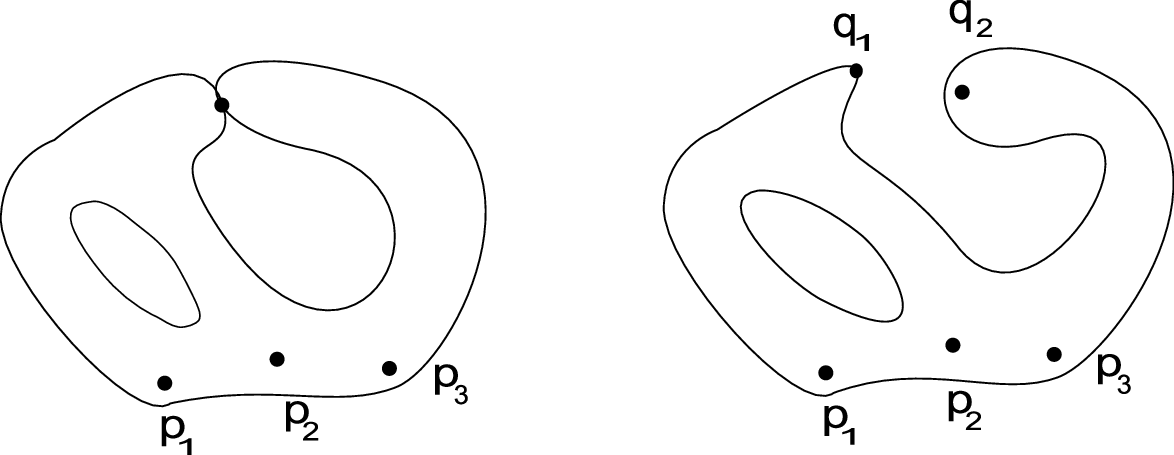}
\caption{Normalization of a  triple marked stable genus $2$ curve.}
\label{fig:4}
\end{figure}

 In Section \ref{sheaf} we build a multiplication operation on the direct sum $\mathcal{V}^{\dagger}(G) = \bigoplus_{\vec{\lambda}, L} \mathcal{V}^{\dagger}(\vec{\lambda}, L)$ from elements of Kac-Moody representation theory.  The following is a consequence of this construction. 

\begin{proposition}\label{T1}
For any simple Lie algebra $\mathfrak{g}$ with associated simple, simply connected
reductive group $G$, the direct sum of vector bundles,

\begin{equation}\mathcal{V}^{\dagger}(G) = \bigoplus_{\vec{\lambda}, L} \mathcal{V}^{\dagger}(\vec{\lambda}, L)\end{equation}

\noindent
has the structure of a flat sheaf of algebras on  $\bar{\mathcal{M}}_{g, n}.$  Over a smooth marked curve $(C, \vec{p}),$ multiplication on this sheaf agrees with multiplication of global sections on the corresponding line bundles over the moduli $\mathcal{M}_{C, \vec{p}}(G).$ 

\begin{equation}
\mathcal{V}_{C, \vec{p}}^{\dagger}(G) \cong Cox(\mathcal{M}_{C, \vec{p}}(G)).\\
\end{equation}
\end{proposition}

We call $\mathcal{V}_{C, \vec{p}}^{\dagger}(G)$ the algebra of conformal blocks over $C, \vec{p}$.   The global object $\mathcal{V}^{\dagger}(G)$ relates the Cox ring of $\mathcal{M}_{C, \vec{p}}(G)$ to the algebra of conformal blocks for a non-smooth, stable curve by a flat family. 

Recall that $\bar{\mathcal{M}}_{g, n}$ is stratified by the stability type of the curves $C, \vec{p},$ this is the source of the graph combinatorics in Theorem \ref{main}.    The strata of $\bar{\mathcal{M}}_{g, n}$ are indexed by connected graphs $\Gamma$ with $n$ labelled leaves. Each internal vertex of the graph is labelled with a number $g_i,$ thought of as the ``internal genus'' of that vertex (see Figure \ref{fig:GRAPH}). The genus $g$ of the whole graph is computed by summing these numbers and adding the first Betti number of $\Gamma$.  An internal vertex corresponds to a smooth component in the normalization of a representative curve of the stratum, and leaves correspond to marked points.  The lowest strata of $\bar{\mathcal{M}}_{g, n}$ are isolated points indexed precisely by trivalent graphs. In Section \ref{filter}, we utilize the factorization rules of conformal blocks to degenerate the algebra over non-smooth, stable curves. 

\begin{proposition}\label{T2}
Let $C, \vec{p}$ be a stable curve of stability type $\Gamma,$ with smooth normalization $\tilde{C}, \vec{p}, \vec{q}_1, \vec{q}_2$. There is a flat degeneration, 

\begin{equation}
\mathcal{V}^{\dagger}_{C, \vec{p}}(G) \Rightarrow [\mathcal{V}^{\dagger}_{\tilde{C}, \vec{p}, \vec{q}_1, \vec{q}_2}(G)]^{T_{\Gamma}}.\\
\end{equation}

Here $T_{\Gamma} = (T \times \C^*)^{|E(\Gamma)| - |\vec{p}|}$, where $T\subset G$ is a maximal torus. 
\end{proposition}

\noindent
Taken together, Propositions \ref{T1} and \ref{T2} prove Theorem \ref{main}. 

\begin{figure}[htbp]
\centering
\includegraphics[scale = 0.3]{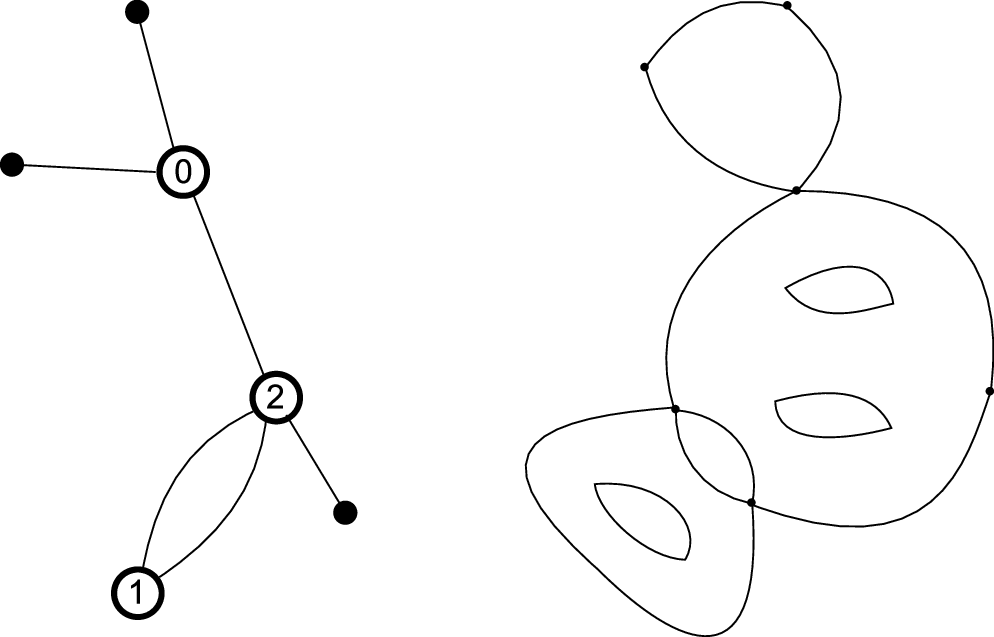}
\caption{The graph of a stable curve type.}
\label{fig:GRAPH}
\end{figure}

\subsection{The relationship with configuration spaces}

There is a natural map relating the space of conformal blocks with labels $\vec{\lambda}$ to the space of $\mathfrak{g}$ invariants in $V(\vec{\lambda}^*) = V(\lambda_1^*)\otimes \ldots \otimes V(\lambda_n^*).$  This is called the correlation map, see the book of Ueno \cite{U} for more information on its properties.  

\begin{equation}F_{C, \vec{p}}: \mathcal{V}_{C, \vec{p}}^{\dagger} (\vec{\lambda}, L) \to Hom_{\C}(V(\vec{\lambda})/ \mathfrak{g} V(\vec{\lambda}), \C)  \cong V(\vec{\lambda^*})^{\mathfrak{g}}\end{equation}

\noindent
 When the genus of the curve $C$ is $0,$ the map $F_{C, \vec{p}}$ is injective, allowing us to make certain aspects of conformal blocks more concrete by relating them to vector spaces from classical representation theory.  Let $U \subset G$ be the maximal unipotent subgroup contained in $B$, and let $\mathfrak{A}_n$ be the algebra invariants in the tensor product $\C[G/U]^{\otimes n}$ with respect to the left diagonal action by $G$.    The space of invariants $V(\vec{\lambda^*})^{\mathfrak{g}}$ is a subspace of $\mathfrak{A}_n$ and is also the space of global sections of a line bundle $\mathcal{L}(\vec{\lambda})$ on the space of configurations $M_{\vec{\lambda}}$, see Subsection \ref{cspaces}. Furthermore, the projective coordinate ring $\mathfrak{A}_{\vec{\lambda}} = \bigoplus_{N \geq 0} H^0(M_{\vec{\lambda}}, \mathcal{L}(\vec{\lambda})^{\otimes N})$ is always a subalgebra of $\mathfrak{A}_n$.   In Sections \ref{sheaf} and \ref{branch} we show that the correlation map can be enhanced to a map of algebras $F_{C, \vec{p}}: \mathcal{V}^{\dagger}_{C, \vec{p}}(G) \to \hat{\mathfrak{A}}_n$, where $\hat{\mathfrak{A}}_n$ is a certain Rees algebra of $\mathfrak{A}_n$ (see Section \ref{branch}). This relationship then passes to a graded inclusion $R_{C, \vec{p}}(\vec{\lambda}, L) \subset \mathfrak{A}_{\vec{\lambda}}$. In Section \ref{branch} we then relate the degenerations constructed by Theorem \ref{main} to degenerations of configuration spaces constructed in \cite{M2} and \cite[Proposition 4.9]{M14}.

\begin{theorem}\label{G0}
Let $(C, \vec{p})$ be a marked, stable, genus $0$ curve of type $\tree.$ 
Then the degeneration from Theorem \ref{T2} on $\mathcal{V}_{C, \vec{p}}^{\dagger}(G)$ extends to a degeneration on $\hat{\mathfrak{A}}_n$ corresponding to $\tree$ under the correlation morphism $F_{C, \vec{p}}.$  Furthermore, the induced degeneration of $R_{C, \vec{p}}(\vec{\lambda}, L)$ extends to a degeneration on $\mathfrak{A}_{\vec{\lambda}}$.
\end{theorem}

  For $\mathfrak{g} = sl_2(\C),$ the algebra $\mathfrak{A}_n$ is the projective coordinate ring of the Grassmannian variety $Gr_2(\C^n)$.  In \cite{SpSt}, Speyer and Sturmfels describe the tropical variety $T^{n, 2}$ of $Gr_2(\C^n)$ with respect to the embedding in $\mathbb{P}^{\binom{n}{2}-1}$ given by the Pl\''ucker generators of its coordinate ring.  They show that $T^{n, 2}$ has a maximal face for each trivalent tree with $n$ ordered leaves.  As $T^{n, 2}$ is a subfan of the Gr\''obner fan of $Gr_2(\C^n)$, each point $w \in T^{n, 2}$ defines an initial ideal $I_w$ of the ideal of relations on the Pl\''ucker generators.  The degeneration of $\mathfrak{A}_n$ associated to a trivalent tree $\tree$ with $n$ ordered leaves constructed in Theorem \ref{G0} is then presented by an initial ideal $I_w$ from the associated face of $T^{n, 2}$ (see \cite{HMM}, \cite{HMSV}).  These degenerations are used by Howard, Millson, Snowden, and Vakil \cite{HMSV} to study the moduli space $M_{\vec{r}}$ of weighted $\vec{r}-$weighted ordered point arrangements on $\mathbb{P}^1$, for $\vec{r} \in \Z_{\geq 0}^n.$  The appearance of trees $\tree$ in the context of degenerations of $M_{\vec{r}}$, and Sturmfels and Xu's work on $\mathcal{M}_{\mathbb{P}^1, \vec{p}}(SL_2(\C))$ led Millson to conjecture the following, \cite{Mil}.

\begin{conjecture}[Millson]
For a curve $C$ of genus $0$ and a trivalent tree $\tree$ with $n$-leaves, consider the degeneration of the ring $\mathfrak{A}_{\vec{r}}$ induced by a  constructed by Speyer and Sturmfels associated to the tree $\tree.$  Then the induced degeneration on
$R_{C, \vec{p}}(\vec{r}, L)$ is isomorphic to $\C[P_{\tree}(\vec{r}, L)].$  
\end{conjecture}

\noindent
The following is a consequence of Theorem \ref{G0}.

\begin{corollary}
For a curve $C$ of genus $0$ and a trivalent tree $\tree$ with $n$-leaves, the induced degeneration on
$R_{C, \vec{p}}(\vec{r}, L)$ above is isomorphic to $\C[P_{\tree}(\vec{r}, L)]$ when $(C, \vec{p})$ is the stable curve of type $\tree.$
\end{corollary}

\subsection{Phylogenetics and conformal blocks} 

Along with their relationship to $SL_2(\C)$ conformal blocks, the affine semigroup algebras $\C[P_{\Gamma}]$ make an appearance
in mathematical biology.  The scheme $Proj(\C[P_{\Gamma}])$ is a statistical model based on the Jukes-Cantor binary model of phylogenetics, see \cite{BW}, \cite{Bu}, \cite{BBKM}, \cite{StXu}.  In \cite{Bu} and \cite{BW}, Buczy\'{n}ska and Wi\'{s}niewski show geometrically that the Hilbert functions of the algebras $\C[P_{\Gamma}]$ only depend on the number of leaves and first Betti number of $\Gamma$.   The visually appealing method employed by Buczy\'{n}ska and Wi\'{s}niewski to obtain this result is to construct pair-wise deformations between algebras associated to combinatorial alterations of the underyling graphs (see Figure \ref{fig:3}). In a sense, Theorem \ref{degensl2} ``fills in'' this picture over $\bar{\mathcal{M}}_{g, n}$. 

\begin{figure}[htbp]
\centering
\includegraphics[scale = 0.35]{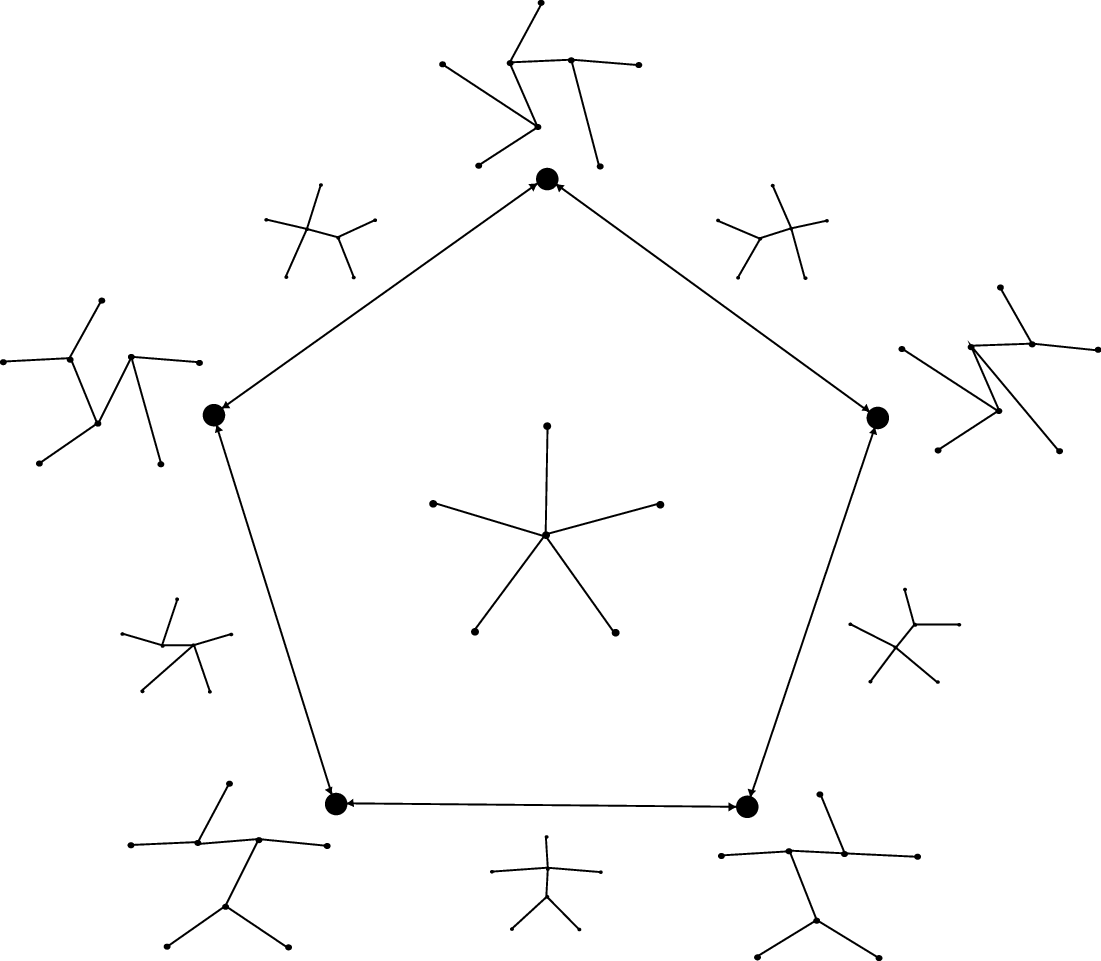}
\caption{Combinatorial alterations between trees with $5$ ordered leaves}
\label{fig:3}
\end{figure}

The commutative algebra of $\C[P_{\Gamma}]$ is also studied by Buczy\'{n}ska, Buczy\'{n}ski, Kubjas, and Micha\l{}ek in \cite{BBKM}, where they prove the following theorem. 

\begin{theorem}[BBKM]
Let $\Gamma$ be a graph with first Betti number $g$ and $n$ leaves, then $\C[P_{\Gamma}]$ is generated
in degree $\leq g+1.$ There exist graphs where this bound is obtained. 
\end{theorem} 

 Although this was not the focus of their work, note that the result in \cite{BBKM} and Theorem \ref{degensl2} imply that $Cox(\mathcal{M}_{C, \vec{p}}(SL_2(\C)))$ is generated by conformal blocks of level $\leq g+1$ for $C, \vec{p}$ generic, and Theorem \ref{sl2gen} shows that this bound can be lowered to $2$. The \cite{BBKM} result also shows that the special fibers $\mathcal{V}_{C, \vec{p}}^{\dagger}(SL_2(\C))$ for $C, \vec{p}$ of trivalent graph type $\Gamma$ can all be generated in degree $\leq g+1.$

 The connection between phylogenetics and moduli of $SL_2(\C)$ principal bundles was first made in the $g = 0$ case by in \cite{StXu}.  Sturmfels and Xu show that the binomial ideal defining $\C[P_{\tree}]$ for $\tree$ a tree coincides with an initial ideal of an ideal presenting the algebra $Cox(\mathcal{M}_{C, \vec{p}}(SL_2(\C)))$.  We note that this connection can also be made directly through the combinatorics of $sl_2(\C)$ conformal blocks. In particular, the quantum Clebsch Gordon rule (see Section \ref{sl2}) and the factorization rules (see Section \ref{filter}), establish directly that the lattice points of the polytope $P_{\Gamma}(\vec{r}, L)$ count the dimension of the space $\mathcal{V}_{C, \vec{p}}^{\dagger}(\vec{r}, L)$ when the genus of $C$ is $\beta_1(\Gamma)$ and $|\vec{p}|$ is the number of leaves of $\Gamma$.

\subsection{Remarks on integrable systems in the moduli of principal bundles}

The polytopes $P_{\Gamma}$ and $P_{\Gamma}(\vec{r}, L)$ appear in work of Hurtubise, Jeffrey \cite{HJ} and Jeffrey, Weitsman \cite{JW}, on the integrable systems in the moduli of bundles associated to the Goldman flows on those spaces. In particular, in \cite{HJ}, Hurtubise and Jeffrey note that the presence of a dense, open integrable system in the moduli space with momentum image $P_{\Gamma}$ almost gives a proof of the Verlinde formula.  They reason that if the moduli space were toric, then the Verlinde formula could be computed by counting the lattice points in $P_{\Gamma}$, which coincides with the expected dimension.  Theorem \ref{degensl2} enhances this picture by showing that the toric variety associated to $P_{\Gamma}$ is a flat degeneration of the moduli space.  It would be interesting to relate the degeneration constructed in Theorem \ref{degensl2} and the integrable system studied in \cite{HJ} along the lines 
of recent work of Kaveh and Harada \cite{HK} (see also Subsection \ref{NOKbody}).

\subsection{Notation}
Here we collect some frequently used notation. \\

\begin{tabular}{l r}
$G$ & A simple, simply connected affine group over $\C.$\\
$\mathfrak{g}$ & The Lie algebra of $G$.\\
$\lambda \in \Delta$ & A dominant weight of $\mathfrak{g}$ in a Weyl chamber.\\ 
$V(\lambda)$ & The irreducible representation of $\mathfrak{g}$ associated to $\lambda.$\\
$\hat{\mathfrak{g}}$ & The affine Kac-Moody algebra of $\mathfrak{g}.$\\
$\mathcal{H}(\lambda, L)$ & The integrable highest weight representation of $\hat{\mathfrak{g}}$ associated to $\lambda, L.$\\
$(C, \vec{p})$ & A curve with marked points.\\
$\mathcal{V}_{C, \vec{p}}^{\dagger}(\vec{\lambda}, L)$ & The space of conformal blocks associated to the data $C, \vec{p}, \vec{\lambda}, L.$\\
$\mathcal{V}_{C, \vec{p}}^{\dagger}(G)$ & The algebra of conformal blocks on $[C, \vec{p}]$ for the group $G.$\\
$\mathcal{M}_{C, \vec{p}}(G)$ & The moduli stack of quasi-parabolic principal $G-$bundles on $[C, \vec{p}]$.\\
$M_{C, \vec{p}}(\vec{\lambda}, L)$ & The moduli space of semistable parabolic $G-$bundles on $[C, \vec{p}]$.\\
$\bar{\mathcal{M}}_{g, n}$ &The moduli stack of stable curves of genus $g$ with $n$ marked points.\\
$\Gamma$ & A graph.\\
$E(\Gamma)$ & The edges of $\Gamma$.\\
$V(\Gamma)$ & The non-leaf vertices of $\Gamma$.\\
\end{tabular}

\bigskip

\section{The sheaf of conformal blocks}\label{sheaf}

In this section we construct the multiplication operation on the sheaf of conformal blocks, and show that its specialization
at smooth marked curve $(C, \vec{p})$ is equal to multiplication of global sections of line bundles on the moduli $\mathcal{M}_{C, \vec{p}}(G)$.
We thank Eduard Looijenga for the remarks he provided on his construction of the sheaf $\mathcal{V}^{\dagger}(\vec{\lambda}, L)$ of conformal blocks on $\bar{\mathcal{M}}_{g, n}$. For a simple Lie algebra $\mathfrak{g}$ we fix a Cartan subalgebra, a system of positive roots, and we let $\Delta$ denote the corresponding Weyl chamber.

\subsection{Construction of the sheaf of conformal blocks}

We refer the reader to the accounts of this construction in \cite{B}, \cite{K}, \cite{KNR}, \cite{L}, \cite{SU}, \cite{TUY} and \cite{Fak}. Let $\theta$ be the longest root, with associated Cartan element $\theta^{\vee}$; the level $L$ alcove $\Delta_L \subset \Delta$ is the simplex defined by the equation $\lambda(\theta^{\vee}) \leq L$.   For each dominant weight $\lambda \in \Delta_L$ there is an integrable highest weight module $\mathcal{H}(\lambda, L)$ of the affine Kac-Moody algebra $\hat{\mathfrak{g}}$ (see \cite[1.5]{B}). Recall that $\hat{\mathfrak{g}}$ contains a subalgebra naturally isomorphic to $\mathfrak{g}.$  The module $\mathcal{H}(\lambda, L)$ likewise contains a $\mathfrak{g}$-highest weight vector $v_{\lambda}$ which generates the $\mathfrak{g}$ irreducible representation $V(\lambda)$. 

The Kac-Moody algebra $\hat{\mathfrak{g}}$ is a central extension of $\mathfrak{g}\otimes \C((t))$, where the additional central element acts on $\mathcal{H}(\lambda, L)$ with weight $L$.  Let  $\hat{\mathfrak{g}}_n$ be the Lie algebra $\sum_{i = 1}^n \hat{\mathfrak{g}}$ with central elements identified, this algebra acts on tensor products of integrable highest weight modules with a common level: $\mathcal{H}(\vec{\lambda}, L) = \mathcal{H}(\lambda_1, L) \otimes \ldots \otimes \mathcal{H}(\lambda_n, L)$.   For a connected, stable curve $(C, \vec{p})$, there is an associated Lie algebra $\hat{\mathfrak{g}}[C, \vec{p}] = \mathfrak{g}\otimes \C[C \setminus \{\vec{p}\}]$. By fixing a local parameter $t_i$ at each marked point $p_i$ we obtain a map $\hat{\mathfrak{g}}[C, \vec{p}] \to \bigoplus_{i = 1}^n \mathfrak{g} \otimes \C((t_i))$ by power-series expansion. By the Residue Theorem (see \cite[Part I]{B}), this map extends to the central extension $\hat{\mathfrak{g}}_n$, so $\mathcal{H}(\vec{\lambda}, L)$ can be considered as a representation of $\hat{\mathfrak{g}}[C, \vec{p}]$.

Now we sheafify this construction following \cite{L} and \cite{Fak}. Let $S$ be a smooth affine variety over $\C.$ By a stable curve of genus $g$ over $S$ with $n$ marked points we mean a proper, flat map $\pi:C \to S$ with fibers equal to genus $g$ curves with at worst double point singularities, and $n$ pairwise non-intersecting sections $p_1, \ldots, p_n: S \to C$ with images contained in the smooth locus, such that $C \setminus \cup p_i(S)$ is affine over $S.$  We assume we have specified isomorphisms $\psi_i: \hat{\mathcal{O}}_{C, p_i(S)} \to H^0(S, \mathcal{O}_S)[[t]]$.

  We require the following sheaves of Lie algebras over $\mathcal{O}_S$:  $\hat{\mathfrak{g}}_n(S) = \hat{\mathfrak{g}}_n \otimes \mathcal{O}_S$,  $\mathfrak{g}(S) = \mathfrak{g} \otimes \mathcal{O}_S$, and $\hat{\mathfrak{g}}(C, \vec{p}) = \mathfrak{g} \otimes \pi_*[\mathcal{O}_{C \setminus \cup p_i(S)})].$  The algebra $\mathfrak{g}(S)$ can be realized as a sub Lie-algebra of $\hat{\mathfrak{g}}(C, \vec{p}),$ and the fiber of $\hat{\mathfrak{g}}(C, \vec{p})$ at $s$ is equal to $\hat{\mathfrak{g}}[\pi^{-1}(s), \vec{p}(s)].$  Furthermore, by using the $\psi_i$, we realize $\hat{\mathfrak{g}}(C, \vec{p})$ as a sub Lie algebra of $\hat{\mathfrak{g}}_n(S)$.  We also require the sheafified highest weight representations: $\mathcal{H}_S(\vec{\lambda}, L) =$ $\mathcal{H}(\vec{\lambda}, L) \otimes \mathcal{O}_S$, $V_S(\vec{\lambda}) = V(\vec{\lambda}) \otimes \mathcal{O}_S.$  The sheaf $\mathcal{H}_S(\vec{\lambda}, L)$ is a $\hat{\mathfrak{g}}_n(S)$ module, and therefore also a $\hat{\mathfrak{g}}(C, \vec{p})$ module, and the sheaf $V_S(\vec{\lambda})$ is likewise a  $\mathfrak{g}(S)$ submodule of $\mathcal{H}_S(\vec{\lambda}, L).$ 

For any Lie algebra $\mathfrak{g}$ over $\C$ and representation $M$ there is a space invariants:

\begin{equation}Hom_{\C}(M / \mathfrak{g} M, \C) \cong
(M^*)^{\mathfrak{g}}.\end{equation}

\noindent
The same construction may be applied to sheaves of Lie algebras and representations
over a scheme $S,$ with $Hom_{\mathcal{O}_S}(-,-),$ the sheaf
of morphisms, and $\mathcal{O}_S$ as a dualizing object. 

\begin{definition}
The sheaf of vacua or conformal blocks, $\mathcal{V}_{C, \vec{p}}^{\dagger}(\vec{\lambda}, L)$ is defined to be the following sheaf of invariants. 

\begin{equation}\mathcal{V}_{C, \vec{p}}^{\dagger}(\vec{\lambda}, L) = Hom_{\mathcal{O}_S}(\mathcal{H}_S(\vec{\lambda}, L)/\hat{\mathfrak{g}}(C, \vec{p}) \mathcal{H}_S(\vec{\lambda}, L), \mathcal{O}_S)\end{equation}
\end{definition}
\noindent

Note that $\mathcal{V}_{C, \vec{p}}^{\dagger}(\vec{\lambda}, L)$ is naturally a subsheaf of the dual $\mathcal{H}_S(\vec{\lambda}, L)^*.$ Taking a single fiber $\pi^{-1}(s)$ of $\pi$ we have the vector space of conformal blocks,

\begin{equation}\mathcal{V}_{\pi^{-1}(s), \vec{p}(s)}^{\dagger}(\vec{\lambda}, L) = Hom_{\C}(\mathcal{H}(\vec{\lambda}, L) /
\hat{\mathfrak{g}}[\pi^{-1}(s), \vec{p}(s)] \mathcal{H}(\vec{\lambda}, L), \C).\end{equation}

\noindent

Over a general smooth base scheme one can always choose the isomorphisms $\psi_i$ Zariski locally.  Furthermore, a description of these sheaves can be given which does not depend on the choice of the $\psi_i$, see \cite{L} and \cite[1.7]{B}.  The corresponding sheaves $\mathcal{V}^{\dagger}(\vec{\lambda}, L)$ on the moduli stack $\bar{\mathcal{M}}_{g, n}$  are proved to be locally free and coherant in \cite{U} and \cite{TUY}, see also \cite{L}.

\subsection{Multiplication of conformal blocks}\label{D2}

Now we define the multiplication operation on sheaves of conformal blocks.  We let $C_{\lambda ; \gamma}: V(\lambda + \gamma) \to V(\lambda)\otimes V(\gamma)$ be the $\mathfrak{g}$  intertwiner which sends the highest weight vector $v_{\lambda+\gamma}$ to $v_{\lambda}\otimes v_{\gamma}.$  This operation also makes sense on integrable $\hat{\mathfrak{g}}$ representations, so by abuse of notation $C_{\lambda ; \gamma}: \mathcal{H}(\lambda + \gamma, L + K) \to \mathcal{H}(\lambda, L)\otimes \mathcal{H}(\gamma, K)$ is the map defined in the same way (see \cite[1.6]{K}).  Notice that these maps coincide under the restriction to the subspace $V(\lambda + \gamma) \subset \mathcal{H}(\lambda + \gamma, L + K).$   Furthermore, the inclusions $i: V(\lambda) \to \mathcal{H}(\lambda, L)$ and $\mathfrak{g} \subset \hat{\mathfrak{g}}[C, \vec{p}]$ induce the correlation maps $F: \mathcal{V}_{C, \vec{p}}^{\dagger}(\vec{\lambda}, L) \to [V(\vec{\lambda})]^{\mathfrak{g}}$ We let $\bar{I}$ denote the natural inclusion $\bar{I}: \mathcal{V}_{C, \vec{p}}^{\dagger}(\vec{\lambda}, L) \to \mathcal{H}_S(\vec{\lambda}, L)^*$, and $I$ denote the inclusion $I: [V_S(\vec{\lambda})^*]^{\mathfrak{g}(S)} \to V_S(\vec{\lambda})^*$. As a consequence of these definitions the following diagram commutes:

$$
\begin{CD}
\mathcal{V}^{\dagger}_{C, \vec{p}}(\vec{\lambda}, L) @>\bar{I}>> \mathcal{H}_S(\vec{\lambda}, L)^*\\
@A F AA @A\vec{i}^*AA\\
[V_S(\vec{\lambda})^*]^{\mathfrak{g}(S)} @>I>> V_S(\vec{\lambda})^*.\\
\end{CD}
$$

\noindent
The map $C_{\vec{\lambda} ; \vec{\gamma}}^*: \mathcal{H}_S(\vec{\lambda}, L)^* \otimes_{\mathcal{O}_S} \mathcal{H}_S(\vec{\gamma}, K)^* \to \mathcal{H}_S(\vec{\lambda} + \vec{\gamma}, K + L)^*$ formed by dualizing $C_{\vec{\lambda} ; \vec{\gamma}}$  is the global section multiplication operation on the projective coordinate rings of Kac-Moody flag varieties by the Kac-Moody version of the Borel-Bott-Weil theorem of Kumar, \cite{K}.  As a consequence this map defines a commutative and associative multiplication on the direct sum of the $\mathcal{H}_S(\vec{\lambda}, L)^*$, and restricts to such an operation on the sheaves of vacua.  We define $C_{\vec{\lambda} ; \vec{\gamma}}^*: [V_S(\vec{\lambda})]^{\mathfrak{g}(S)} \otimes [V_S(\vec{\gamma})]^{\mathfrak{g}(S)}$ $\to [V_S(\vec{\lambda} + \vec{\gamma})]^{\mathfrak{g}(S)}$ using the same recipe, the following lemma is then immediate.

\begin{lemma}\label{mainsquare}
The multiplication maps $C_{\vec{\lambda} ; \vec{\gamma}}^*$ commute with the correlation map $F$. 
\end{lemma}

 When the genus $g = 0$, the map $F$ is a monomorphism by an observation of Tsuchiya, Ueno and Yamada in \cite{TUY}.  Everything here commutes with specialization, so we also obtain the following diagram: 

$$
\begin{CD}
 \mathcal{V}_{\pi^{-1}(s), \vec{p}(s)}^{\dagger}(\vec{\lambda}, L) \otimes \mathcal{V}_{\pi^{-1}(s), \vec{p}(s)}^{\dagger}(\vec{\gamma}, K) @>C_{\vec{\lambda} ; \vec{\gamma}}^*>>\mathcal{V}_{\pi^{-1}(s), \vec{p}(s)}^{\dagger}(\vec{\lambda} + \vec{\gamma}, L + K)\\
@V F \otimes F VV @V F VV \\
[V(\vec{\lambda})^*]^{\mathfrak{g}} \otimes [V(\vec{\gamma})^*]^{\mathfrak{g}} @>C_{\vec{\lambda} ; \vec{\gamma}}^*>> [V(\vec{\gamma} + \vec{\lambda})^*]^{\mathfrak{g}}.\\
\end{CD}
$$\\

This multiplication operation also works well with the alternative constructions of the spaces of conformal blocks given in \cite{B}.  We fix an $n+1$ marked curve $(C, \vec{p}, q).$  By identifying highest weight vectors, we get the following diagram of $\mathfrak{g}$ representations.

$$
\begin{CD}
V(0) \otimes \mathcal{H}(\vec{\lambda},L) @>>> \mathcal{H}(0, L) \otimes \mathcal{H}(\vec{\lambda}, L) @<<< \mathcal{H}(0, L) \otimes V(\vec{\lambda})\\
\end{CD}
$$
The space on the left is a $\hat{\mathfrak{g}}[C, \vec{p}]$ 
representation, the middle is a $\hat{\mathfrak{g}}[C, \vec{p}, q]$ 
representation, and the space on the right is a $\hat{\mathfrak{g}}[C,  q]$ representation where
the action on $V(\lambda_i)$ is by evaluation at $p_i$. The following is a ringification of a result which appears in \cite{B}.

\begin{proposition}\label{reform}
Let $C$ be a stable curve. 
The following are isomorphisms of algebras over $\C.$
$$
\begin{CD}
\bigoplus_{\vec{\lambda}, L} \mathcal{V}_{C, \vec{p}}^{\dagger}(\vec{\lambda}, L) @>>> \bigoplus_{\vec{\lambda}, L} \mathcal{V}_{C, \vec{p}, q}^{\dagger}(0, \vec{\lambda}, L) @<<< \bigoplus_{\vec{\lambda}, L} [\mathcal{H}(0, L) \otimes V(\vec{\lambda})]^{\mathfrak{g}[C,  q]}\\
\end{CD}
$$
\end{proposition}

\begin{proof}
By a theorem in \cite{B} the morphism on the right is an isomorphism of vector spaces, and by $vacuum$ $propagation$ (see \cite{TUY}, \cite{SU}, \cite{B} and \cite{NT}) the morphism on the left is also an isomorphism of vector spaces.  Both maps are defined by identifying highest weight vectors, then dualizing, this gives a diagram of rings, with graded components,

$$
\begin{CD}
[V(0) \otimes \mathcal{H}(\vec{\lambda})]^* @<<< [\mathcal{H}(0, L) \otimes \mathcal{H}(\vec{\lambda}, L)]^* @>>> [\mathcal{H}(0, L) \otimes V(\vec{\lambda})]^*\\
\end{CD}
$$
Taking Lie algebra invariants picks out subspaces of these spaces which are preserved by multiplication, these are isomorphisms of algebras.
\end{proof}

\noindent
Beauville's result gives the identification $\mathcal{V}^{\dagger}_{C, \vec{p}}(\vec{\lambda}, L) = [\mathcal{H}(0, L)^* \otimes V(\vec{\lambda})^*]^{\hat{\mathfrak{g}}[C , q]},$ see also \cite{LS}.

\subsection{Moduli of quasi-parabolic principal bundles}

For what follows we refer the reader to the work of 
Kumar, Kumar-Narasimhan-Ramanathan, Laszlo-Sorger, and Pauly 
(\cite{K}, \cite{KNR}, \cite{LS}, \cite{S}, \cite{P}). 
The theorem below can be found in \cite{KNR}, \cite{LS} and in \cite{S} for the exceptional groups.

\begin{theorem}
The moduli stack of quasi-parabolic $G$-bundles on $C$ smooth, with parabolic structure $\vec{\Lambda}$ at the marked points $\mathcal{M}_{C, \vec{p}}(\vec{\Lambda})$ carries a line bundle $\mathcal{L}(\vec{\lambda}, L)$, where $\lambda_i$ a dominant weight in the face of $\Delta$ associated to $\Lambda_i.$  Global sections of this line bundle are are identified with a space of conformal blocks:

\begin{equation}H^0(\mathcal{M}_{C, \vec{p}} (\vec{\Lambda}), \mathcal{L}(\vec{\lambda}, L)) \cong [\mathcal{H}(0, L)^*\otimes V(\vec{\lambda^*})]^{\mathfrak{g}[C, q]}.\end{equation}
\end{theorem}

The stack $\mathcal{M}_{C, \vec{p}}(\vec{\Lambda})$ is obtained
as a quotient of the $ind$-variety 
$Q \times G/\Lambda_1 \times \ldots \times G/\Lambda_n,$ by 
the $ind$-group $G(\C[C \setminus q])$ for $q \in C$, where $Q$ is the affine Grassmannian variety.  This  space is constructed as a quotient $Q = L(G) / L^{+}(G)$, where $L(G)$ is the loop group of $G.$   Let $\hat{\mathcal{O}}_q$
be the formal completion of the local ring at $q,$ and let $\mathfrak{k}_q$ be
the quotient field of $\hat{\mathcal{O}}_q.$  Then $L(G) = G(\mathfrak{k}_q),$ and $L^+(G) = G(\hat{\mathcal{O}}_q).$    The space $Q \times G/ \Lambda_1 \times \ldots \times G/ \Lambda_n$ carries line bundles $L(L, \vec{\lambda})$ with global section spaces
equal to $\mathcal{H}(0, L)^* \otimes V(\vec{\lambda}^*).$ By the Borel-Bott-Weil theorem in the Kac-Moody setting proved by Kumar \cite{K}, multiplication of global sections is computed with the maps $C_{\lambda ; \gamma}^*$  from the previous subsection.

\begin{proposition}\label{moduli}
For $(C, \vec{p}) \in \mathcal{M}_{g, n} \subset \bar{\mathcal{M}}_{g, n}$ there is a monomorphism of multigraded rings

\begin{equation}h_{\vec{\Lambda}}: Cox(\mathcal{M}_{C, \vec{p}}(\vec{\Lambda})) \to \mathcal{V}_{C, \vec{p}}^{\dagger}(G).\end{equation}
The image of this monomorphism is the direct sum of conformal blocks
$\mathcal{V}_{C, \vec{p}}^{\dagger}(\vec{\lambda}, L)$ with $\lambda_i$ a dominant weight
in the face of $\Delta$ associated to $\Lambda_i.$  
This is an isomorphism when all $\Lambda_i$ are Borel subgroups. 
\end{proposition}

\begin{proof}
 As we mentioned in the introduction, in \cite{LS} Laszlo and Sorger identify the Picard group of $\mathcal{M}_{C, \vec{p}}(\vec{\Lambda}):$

\begin{equation}Pic(\mathcal{M}_{C, \vec{p}}(\vec{\Lambda})) = \mathcal{X}(\Lambda_1) \times \ldots \times \mathcal{X}(\Lambda_n) \times \Z, \end{equation}

\noindent
where $\mathcal{X}(\Lambda_i)$ is the character group of $\Lambda_i.$   For any line bundle
$\mathcal{L}(\vec{\lambda}, L)$ on $\mathcal{M}_{C, \vec{p}}(\vec{\Lambda})$
there is an isomorphism between the sections of $\mathcal{L}(\vec{\lambda}, L)$
and the $G(\C[C \setminus q])$-equivariant sections of the pullback bundle
on $Q \times G/\Lambda_1 \times \ldots \times G/\Lambda_n$ 
by a standard theorem on quotient stacks, see \cite{LS}. By Borel-Bott-Weil (standard and Kac-Moody versions), it follows that such a line bundle is effective only if each $\lambda_i$ is dominant and $L$ is non-negative.  In this case we have concluded above that the global sections are spaces of conformal blocks and that the multiplication operation on the section spaces is computed by the multiplication operation in $ \mathcal{V}_{C, \vec{p}}^{\dagger}(G)$.
\end{proof}

\section{Filtrations of the algebra of conformal blocks}\label{filter}

In the previous section we built the flat sheaf of algebras $\mathcal{V}^{\dagger}(G).$ 
This allows us to relate $Cox(\mathcal{M}_{C, \vec{p}}(\vec{\Lambda}))$
to the algebra $\mathcal{V}^{\dagger}_{C', \vec{q}}(G)$ for $(C', \vec{q})$ a singular curve.
In this section we use the factorization map of Tsuchiya-Ueno-Yamada to define a degeneration $\mathcal{V}_{C, \vec{p}}^{\dagger}(G) \Rightarrow [\bigotimes_{v \in V(\Gamma)} \mathcal{V}_{0, 3}^{\dagger}(G)]^{T_{\Gamma}}$ for a singular curve $C, \vec{p}$ of type $\Gamma,$ and prove Theorem \ref{main}.   We begin with a discussion of the factorization isomorphism. For each dominant weight $\lambda \in \Delta$ and its dual $\lambda^*$, and a choice of highest weight vector $v_{\lambda} \in V(\lambda)$ and lowest weight vector $\hat{v}_{\lambda^*}$, let $F_{\lambda}: V(\lambda)\otimes V(\lambda^*) \to \C$  be the unique equivariant map such that $F_{\lambda}(v_{\lambda} \otimes \hat{v}_{\lambda^*}) = 1$.  The map $F_{\lambda}$ gives an isomorphism  

\begin{equation}
Hom_{\C}(V(\lambda), V(\lambda)) \cong V(\lambda)\otimes V(\lambda^*),\\
\end{equation} 

\noindent
where $\sum_i x_i\otimes y_i$ acts on $v \in V(\lambda)$ as $\sum_i x_i \otimes F_{\lambda}(y_i \otimes v).$
We choose $O_{\lambda, \lambda^*} \in V(\lambda)\otimes V(\lambda^*)$
to represent the identity under this isomorphism.  This element defines a $\mathfrak{g}-$linear map,

$$
\begin{CD}
V(\vec{\lambda}) @>\rho_{\alpha}>> V(\vec{\lambda})\otimes V(\alpha)\otimes V(\alpha^*),\\
\end{CD}
$$ 

\noindent
 which sends $X$ to $X \otimes O_{\alpha, \alpha^*}.$   The map $\rho_{\alpha}$ also makes sense for integrable highest weight
representations of $\hat{\mathfrak{g}},$ we can define $\rho_{\alpha}: \mathcal{H}(\vec{\lambda}, L) \to \mathcal{H}(\vec{\lambda}, \alpha, \alpha^*, L) = \mathcal{H}(\vec{\lambda}, L) \otimes \mathcal{H}(\alpha, \alpha^*, L)$. 

We fix a stable curve $C$, with singular point $q \in C$, and we let $\tilde{C}$ be the partial normalization of $C$ at $q$ with the two new marked points $q_1, q_2 \in \tilde{C}.$  The modules $\mathcal{H}(\vec{\lambda}, L)$ and $\mathcal{H}(\vec{\lambda}, \alpha, \alpha^*, L)$ are viewed as representations of $\hat{\mathfrak{g}}[C, \vec{p}]$ and $\hat{\mathfrak{g}}[\tilde{C}, \vec{p}, q_1, q_2]$, respectively.  Taking dual spaces, and then invariants by these Lie algebras yields the following map, which is shown to be injective in \cite{TUY}. 

$$
\begin{CD}
\mathcal{V}^{\dagger}_{C, \vec{p}}(\vec{\lambda}, L) @<\hat{\rho}_{\alpha}<<\mathcal{V}^{\dagger}_{\tilde{C}, \vec{p}, q_1, q_2}(\vec{\lambda}, \alpha, \alpha^*, L)\\
\end{CD}
$$

\noindent
This operation can be performed with any finite number of nodal singular points, $\vec{q}.$ Summing over all $\vec{\alpha} \in \Delta_L^m$ gives the factorization isomorphism. As we show below, the extra dominant weight data $\vec{\alpha} \in \Delta_L^m$ brought out by factorization defines a filtration on $\mathcal{V}^{\dagger}_{C, \vec{p}}(G).$

Consider the tensor product decomposition:

\begin{equation}V(\alpha) \otimes V(\beta) \cong \bigoplus W^{\eta}_{\alpha, \beta} \otimes V(\eta),\end{equation}

\noindent
where $W^{\eta}_{\alpha, \beta} = Hom_{\mathfrak{g}}(V(\eta), V(\alpha)\otimes V(\beta)).$ We have the following identities:  

\begin{equation}
V(\alpha)\otimes V(\beta) \otimes V(\alpha^*)\otimes V(\beta^*) \cong\end{equation}
$$Hom_{\C}(V(\alpha)\otimes V(\beta), V(\alpha) \otimes V(\beta))\cong$$
$$Hom_{\C}( \bigoplus W_{\alpha, \beta}^{\eta} \otimes V(\eta), \bigoplus W_{\alpha, \beta}^{\eta} \otimes V(\eta)) 
\cong [\bigoplus W_{\alpha, \beta}^{\eta}\otimes V(\eta)] \otimes [\bigoplus W_{\alpha^*, \beta^*}^{\eta^*} \otimes V(\eta^*)].$$ \\

For any two maps $f \otimes g \in W_{\alpha, \beta}^{\eta} \otimes W_{\alpha^*, \beta^*}^{\eta^*},$ the map $F_{\alpha}\otimes F_{\beta}\circ (f \otimes g) : V(\eta) \otimes V(\eta^*) \to \C$ must be a multiple $F_{\alpha, \beta}^{\eta}(f, g)$ of $F_{\eta}.$ This assignment defines
a bilinear map $F_{\alpha, \beta}^{\eta}: W_{\alpha, \beta}^{\eta} \otimes W_{\alpha^*, \beta^*}^{\eta^*} \to \C.$ We let $I_{\alpha, \beta}^{\eta}$ represent the identity map under the induced isomorphism $W_{\alpha, \beta}^{\eta} \otimes W_{\alpha^*, \beta^*}^{\eta^*} = Hom(W_{\alpha, \beta}^{\eta}, W_{\alpha, \beta}^{\eta}).$  By definition we have $F_{\alpha}\otimes F_{\beta} = \sum F_{\alpha, \beta}^{\eta} \otimes F_{\eta},$ and therefore the following identity. 

\begin{equation}\label{identity}
O_{\alpha, \alpha^*} \otimes O_{\beta, \beta^*} = \sum I_{\alpha, \beta}^{\eta}\otimes  O_{\eta, \eta^*},
\end{equation}

\noindent
We let  $f_{\eta}: W_{\alpha, \beta}^{\eta}\otimes V(\eta) \to V(\alpha) \otimes V(\beta)$ denote the inclusion defined by the direct sum decomposition, with $f_{\eta^*}: W_{\alpha^*, \beta^*}^{\eta^*}\otimes V(\eta^*) \to V(\alpha^*) \otimes V(\beta^*)$ the corresponding maps on the dual representations.  By our choices above, we have $f_{\alpha + \beta} = C_{\alpha ; \beta},$ and it is a straightforward calculation
to verify that $f_{\alpha^* + \beta^*} = C_{\alpha^* ; \beta^*}.$  We recall the $\hat{\mathfrak{g}}$ Verma module $\bar{V}(\lambda, L)$ from \cite[1.6]{B} and \cite[1.5]{K}. The $f_{\eta}$ give maps,

$$
\begin{CD}
W_{\alpha, \beta}^{\eta}\otimes \bar{V}(\eta, K+L) @> \bar{f}_{\eta}>> \mathcal{H}(\alpha, L) \otimes \mathcal{H}(\beta, K),\\
\end{CD}
$$
\noindent
and the identity \ref{identity} above implies that the following diagram commutes:

\begin{equation}\label{D}
\tiny
\begin{CD}
\mathcal{H}(\vec{\lambda}, L) \otimes \mathcal{H}(\vec{\gamma}, K) @>\rho_{\alpha} \otimes \rho_{\beta}>>  [\mathcal{H}(\vec{\lambda}, L)\otimes \mathcal{H}(\alpha, \alpha^*, K)] \otimes [\mathcal{H}(\vec{\gamma}, K) \otimes \mathcal{H}(\beta, \beta^*, L)] \\
@A C_{\vec{\lambda} ; \vec{\gamma}} AA  @A\sum_{\eta} C_{\vec{\lambda} ; \vec{\gamma}} \otimes \bar{f}_{\eta} \otimes \bar{f}_{\eta^*} AA \\
\mathcal{H}(\vec{\lambda} + \vec{\gamma}, K + L) @>\sum_{\eta}(I_{\alpha, \beta}^{\eta} \otimes \rho_{\eta})>>
\bigoplus_{\eta} \mathcal{H}(\vec{\lambda} + \vec{\gamma}, K + L)\otimes W_{\alpha, \beta}^{\eta} \otimes W_{\alpha^*, \beta^*}^{\eta^*} \otimes \bar{V}(\eta, K+L) \otimes \bar{V}(\eta^*, K+L). \\
\end{CD}
\normalsize
\end{equation}

\noindent
Here the sum is over all dominant weights $\eta$ which are smaller than $\alpha + \beta$ in the dominant weight ordering. 
The bottom map in the diagram sends a vector $Y$ to $\sum Y \otimes I_{\alpha, \beta}^{\eta} \otimes O_{\eta, \eta^*},$ so we replace the map $\bar{f}_{\eta} \otimes \bar{f}_{\eta^*}$ with $\phi_{\eta, \eta^*}(X) = \bar{f}_{\eta} \otimes \bar{f}_{\eta^*}(I_{\alpha, \beta}^{\eta} \otimes X).$  For $\eta = \alpha + \beta,$ we have by definition, 

\begin{equation}\bar{f}_{\alpha + \beta} \otimes \bar{f}_{\alpha^* + \beta^*} = \phi_{\alpha + \beta, \alpha^* + \beta^*} = C_{\alpha , \alpha* ; \beta, \beta^*}.\\\end{equation}

\begin{proposition}\label{factor}

The following diagram commutes. 

$$
\begin{CD}
\mathcal{V}_{C, \vec{p}}^{\dagger}(\vec{\lambda}, L) \otimes \mathcal{V}_{C, \vec{p}}^{\dagger}(\vec{\gamma}, K) @<\hat{\rho}_{\alpha}\otimes \hat{\rho}_{\beta}<< \mathcal{V}_{\tilde{C}, \vec{p}, q_1, q_2}^{\dagger}(\vec{\lambda}, \alpha, \alpha^*, L) \otimes \mathcal{V}_{\tilde{C}, \vec{p}, q_1, q_2}^{\dagger}(\vec{\gamma}, \beta, \beta^*, K)\\
@V C_{\vec{\lambda} ; \vec{\gamma}}^* VV @V [\sum_{\eta} C_{\vec{\lambda} ; \vec{\gamma}} \otimes  \phi_{\eta, \eta^*}]^* VV\\
\mathcal{V}_{C, \vec{p}}^{\dagger}(\vec{\lambda} + \vec{\gamma}, L +K) @<\sum_{\eta}   \hat{\rho}_{\eta}<< \bigoplus_{\eta} \mathcal{V}_{\tilde{C}, \vec{p}, q_1, q_2}^{\dagger}(\vec{\lambda} + \vec{\gamma}, \eta, \eta^*, L + K)\\
\end{CD}
$$
\end{proposition}

\begin{proof}
We may dualize Diagram \ref{D} to obtain:

\small
$$
\begin{CD}
[\mathcal{H}(\vec{\lambda}, L)]^* \otimes [\mathcal{H}(\vec{\gamma}, K)]^* @<\hat{\rho}_{\alpha} \otimes \hat{\rho}_{\beta}<<  [\mathcal{H}(\vec{\lambda}, L)\otimes \mathcal{H}(\alpha, \alpha^*, K)]^* \otimes [\mathcal{H}(\vec{\gamma}, K) \otimes \mathcal{H}(\beta, \beta^*, L)]^* \\
@V C^*_{\vec{\lambda} + \vec{\gamma}} VV  @V[\sum_{\eta} C_{\vec{\lambda} ; \vec{\gamma}} \otimes \phi_{\eta,\eta}]^* VV \\
[\mathcal{H}(\vec{\lambda} + \vec{\gamma}, K + L)]^* @<\sum_{\eta} \hat{\rho}_{\eta}<<
\bigoplus_{\eta} [\mathcal{H}(\vec{\lambda} + \vec{\gamma}, K + L)\otimes \bar{V}(\eta, K+L) \otimes \bar{V}(\eta^*, K+L)]^* .\\
\end{CD}
$$\\
\normalsize
We must determine what happens to the spaces of conformal blocks inside each of the vector spaces in this diagram. 
It follows that $ C^*_{\vec{\lambda} + \vec{\gamma}}$ maps 
$\mathcal{V}^{\dagger}_{C, \vec{p}}(\vec{\lambda}, L) \otimes \mathcal{V}^{\dagger}_{C, \vec{p}}(\vec{\gamma}, K)$ to 
$\mathcal{V}^{\dagger}_{C, \vec{p}}(\vec{\lambda} + \vec{\gamma}, K + L)$ and $[\sum_{\eta} C_{\vec{\lambda} ; \vec{\gamma}} \otimes \phi_{\eta,\eta}]^*$
maps $ \mathcal{V}^{\dagger}_{\tilde{C}, \vec{p}, q_1, q_2}(\vec{\lambda},\alpha, \alpha^*,K)\otimes \mathcal{V}_{C, \vec{p}, q_1, q_2}^{\dagger}(\vec{\gamma}, \beta, \beta^*, L)$ to $[\mathcal{H}(\vec{\lambda} + \vec{\gamma}, L+K) \otimes \bar{V}(\eta) \otimes \bar{V}(\eta^*)]^{\mathfrak{g}(C \setminus \vec{p}, q_1, q_2)},$ respectively, because both maps are obtained by dualizing and taking invariants.   Furthermore, the map $\hat{\rho}_{\alpha} \otimes \hat{\rho}_{\beta}$ sends vectors from the space $ \mathcal{V}^{\dagger}_{\tilde{C}, \vec{p}, q_1, q_2}(\vec{\lambda},\alpha, \alpha^*,K)\otimes \mathcal{V}_{C, \vec{p}, q_1, q_2}^{\dagger}(\vec{\gamma}, \beta, \beta^*, L)$ into $\mathcal{V}^{\dagger}_{C, \vec{p}}(\vec{\lambda}, L) \otimes \mathcal{V}^{\dagger}_{C, \vec{p}}(\vec{\gamma}, K)$ by the factorization theorem.

In order to analyze the bottom arrow of the diagram, we first consider the projection map, 

\begin{equation}
\pi_{\eta, K +L}: \bar{V}(\eta, K+L) \to \mathcal{H}(\eta, K+L)\\
\end{equation}

\noindent
of highest weight $\hat{\mathfrak{g}}$ modules. The map $[Id\otimes \pi_{\eta, K+L} \otimes \pi_{\eta^*, K+L}]^* \circ \hat{\rho}_{\eta}$,
which takes $ [\mathcal{H}(\vec{\lambda} + \vec{\gamma}, K + L)\otimes \mathcal{H}(\eta, K+L) \otimes \mathcal{H}(\eta^*, K+L)]^*$
to $[\mathcal{H}(\vec{\lambda} + \vec{\gamma}, K + L)]^*$ is by definition equal to the map used in the proof of the factorization rules, this implies
that it takes the space $\mathcal{V}_{\tilde{C}, \vec{p}, q_1, q_2}^{\dagger}(\vec{\lambda} + \vec{\gamma}, \eta, \eta^*, L + K)$ to the space 
$\mathcal{V}_{C, \vec{p}}^{\dagger}(\vec{\lambda} + \vec{\gamma}, L +K)$.  The picture is then completed by a theorem of 
Beauville \cite{B} (proof of Proposition $2.3$), which asserts the following equality for any smooth curve $\tilde{C},$ induced by the maps $\pi_{\eta, K+L}$:

\begin{equation}V_{\hat{\mathfrak{g}}[\tilde{C}, \vec{p}, \vec{q}]}^{\dagger}[\mathcal{H}(\vec{\alpha}, L)\otimes \bar{V}(\vec{\beta}, L)] \cong
V_{\hat{\mathfrak{g}}[\tilde{C}, \vec{p}, \vec{q}]}^{\dagger}[\mathcal{H}(\vec{\alpha}, L) \otimes \mathcal{H}(\vec{\beta}, L)] =
\mathcal{V}^{\dagger}_{\tilde{C}, \vec{p}, \vec{q}}(\vec{\alpha}, \vec{\beta}, L).\end{equation}
\end{proof}

This diagram only represents the case of a curve with one singularity, but the general case follows by the same methods. For any tensor product of elements $\chi_1 \otimes \chi_2 \in \mathcal{V}_{\tilde{C}, \vec{p}, q_1, q_2}^{\dagger}(\vec{\lambda}, \alpha, \alpha^*, L)\otimes \mathcal{V}_{\tilde{C}, \vec{p}, q_1, q_2}^{\dagger}(\vec{\gamma}, \beta, \beta^*, K)$, the multiplication $\chi_1\times \chi_2$ can be expanded as follows:

\begin{equation}\label{product}\chi_1 \times \chi_2 = C_{\vec{\lambda} ; \vec{\gamma}}^*(\hat{\rho}_{\alpha}(\chi_1) \otimes \hat{\rho}_{\beta}(\chi_2)) =   \hat{\rho}_{\alpha+ \beta}\circ C_{\vec{\lambda} ; \vec{\gamma}, \alpha , \beta, \alpha^* , \beta^*}^*(\chi_1 \otimes \chi_2) + \sum \chi_{\eta}.\end{equation}

\noindent
Recall that there is a partial ordering $\prec$ on dominant weights such that $\lambda \prec \gamma$ if and only if $\gamma - \lambda$ is a sum of positive roots.    The $\chi_{\eta}$ in the expansion of $\chi_1\times \chi_2$  are the summands from components of the direct sum with $\eta \prec \alpha + \beta$ as dominant weights, in particular multiplication in $\mathcal{V}_{C, \vec{p}}^{\dagger}$ is multiplication in $\mathcal{V}_{\tilde{C}, \vec{p}, q_1, q_2}^{\dagger}(G)$ with additional ``lower'' terms. We use this observation to build an algebra filtration on $\mathcal{V}_{C, \vec{p}}^{\dagger}(G)$.  Recall that any coweight $\theta$ in the dual Weyl chamber $\Delta^{\vee}$ of $\Delta$ has the property that $\theta(\gamma) \geq \theta(\lambda)$ when $\lambda \prec \gamma$, and furthermore, this inequality is strict when $\theta$ is taken from the interior of $\Delta^{\vee}$.

\begin{definition}
We define a coweighting of the normalized curve $(\tilde{C}, \vec{p}, \vec{q})$ to be an assignment of coweights from the dual Weyl chamber $\theta \in \Delta^{\vee}$ to the new marked points $\vec{q}$, such that identified points $q_1, q_2$ are assigned dual coweights (see Figure \ref{fig:coweight}).
We define a weighting of $(\tilde{C}, \vec{p}, \vec{q})$ analogously. 
\end{definition}

\begin{figure}[htbp]
\centering
\includegraphics[scale = 0.4]{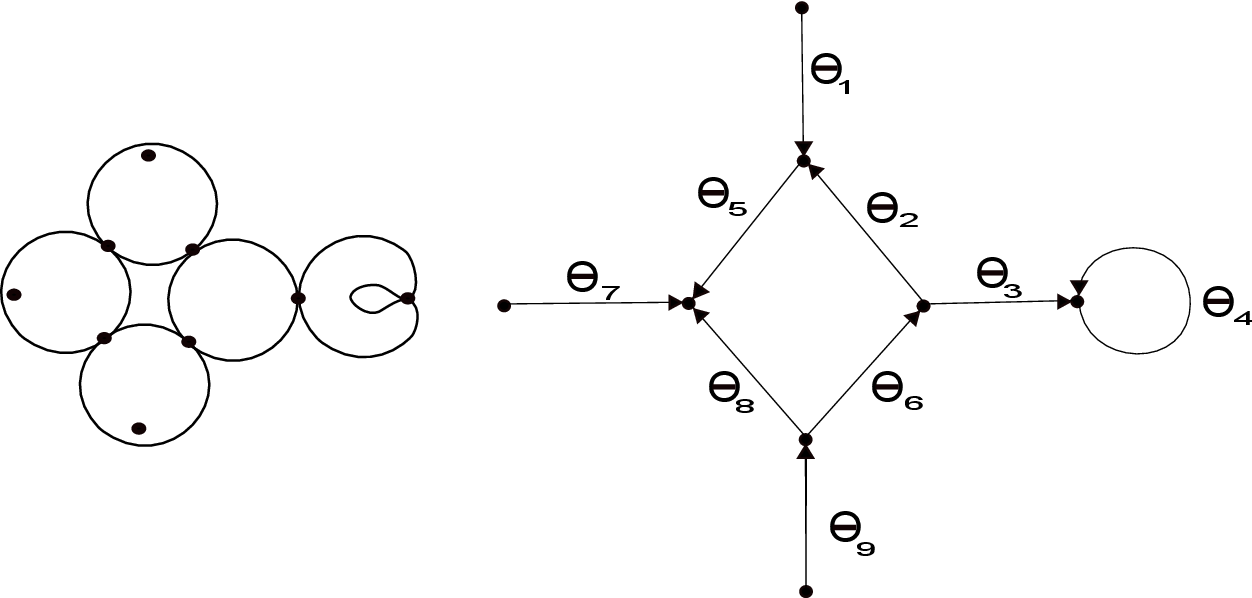}
\caption{A coweighting}
\label{fig:coweight}
\end{figure}

Coweightings and coweightings of $(\tilde{C}, \vec{p}, \vec{q})$ can be visualized as assignments of coweights to the non-leaf edges of the graph $\Gamma$ which labels the stratum of $\bar{\mathcal{M}}_{g, n}$ containing $C$.  For any non-leaf edge $e \in E(\Gamma)$ there are two points $q_1, q_2 \in \vec{q}$, and therefore two dual weights $\alpha, \alpha^*$ assigned to this edge.  We can encode this data by specifying an orientation $q_1 \to q_2$ for each $e \in E(\Gamma)$ and recording one of the weights; it is then understood that the weight recorded is that on $q_1$, and the weight on $q_2$ is obtained by duality.   Now we can denote these objects by a pair $(\Gamma, \vec{\theta})$, where $\Gamma$ is the graph with a choice of orientation, and $\vec{\theta}$ is an assignment of a single coweight (respectively weight) to each non-leaf edge in $\Gamma$.   By extending the natural pairing between weights and coweights bilinearly we can view coweightings of $(\tilde{C}, \vec{p}, \vec{q})$ as linear functions on weightings of $(\tilde{C}, \vec{p}, \vec{q})$:

\begin{equation}
(\Gamma, \vec{\theta})\circ(\Gamma, \vec{\alpha}) = \sum_{e \in Edge(\Gamma)} \theta_e(\alpha_e).\\
\end{equation}

Now we describe a class of filtrations on the algebra of conformal blocks.  Fix a partial normalization $(\tilde{C}, \vec{p}, \vec{q})$ of $(C, \vec{p})$, and choose once and for all an orientation of the edges of the associated $\Gamma$.  Let $\mathcal{V}_{C, \vec{p}}^{\dagger}(\vec{\lambda}, L)_{\vec{\alpha}} \subset \mathcal{V}_{C, \vec{p}}^{\dagger}(\vec{\lambda}, L)$ denote the subspace $\mathcal{V}_{\tilde{C}, \vec{p}, \vec{q}}^{\dagger}(\vec{\lambda}, \vec{\alpha}, \vec{\alpha}^*, L)$ obtained by factorization.  Notice that there is a natural weighting of  $(\tilde{C}, \vec{p}, \vec{q})$ associated to each of these subspaces, namely $(\Gamma, \vec{\alpha})$.   Now, by choosing a coweighting $(\Gamma, \vec{\theta})$, and using the pairing above, we obtain a filtration $\mathcal{F}_{\Gamma, \vec{\theta}}$ on $\mathcal{V}_{C, \vec{p}}^{\dagger}(G)$ by giving $\chi \in  \mathcal{V}_{C, \vec{p}}^{\dagger}(\vec{\lambda}, L)_{\vec{\alpha}}$ filtration level $(\Gamma, \vec{\theta})\circ(\Gamma, \vec{\alpha})$.   

By considering Equation \ref{product} above, we see that the product $\chi_1 \times \chi_2$ and $\hat{\rho}_{\alpha + \beta}\circ C_{\vec{\lambda}, \alpha, \alpha^*;  \vec{\gamma},  \beta, \beta^*}^*(\chi_1 \otimes \chi_2)$ always have the same filtration level, whereas the filtration level of a lower summand $\chi_{\eta}$ is always less than or equal to this value.  For coweightings where all coweights have been chosen in the interior of the dual Weyl chamber of $\mathfrak{g},$ the terms $\chi_{\eta}$ are always given a strictly smaller filtrational level than  $\chi_1 \times \chi_2$. 

\begin{proposition}\label{filterblocks}
The filtration $\mathcal{F}_{\Gamma, \vec{\theta}}$ for $(\Gamma, \vec{\theta})$ respects multiplication on 
the ring $\mathcal{V}_{C, \vec{p}}^{\dagger}(G).$  If the components of $\vec{\theta}$ are 
strictly positive on all positive roots, then the image of the associated graded multiplication map,

\begin{equation}\mathcal{V}_{C, \vec{p}}^{\dagger}(\vec{\lambda}, L)_{\vec{\alpha}} \otimes \mathcal{V}^{\dagger}_{C, \vec{p}}(\vec{\gamma}, K)_{\vec{\beta}} \to \mathcal{F}_{\Gamma, \vec{\theta}}^{\leq \vec{\theta}(\vec{\lambda}, \vec{\alpha} +\vec{\beta}, \vec{\alpha}^* +\vec{\beta}^*)}(\mathcal{V}_{C, \vec{p}}^{\dagger}(G)) /  \mathcal{F}_{\Gamma, \vec{\theta}}^{< \vec{\theta}(\vec{\lambda}, \vec{\alpha} + \vec{\beta}, \vec{\alpha}^* + \vec{\beta}^*)}(\mathcal{V}_{C, \vec{p}}^{\dagger}(G))\end{equation}
can be canonically identified with $\mathcal{V}_{C, \vec{p}}^{\dagger}(\vec{\lambda} + \vec{\gamma}, K+L)_{\vec{\alpha} + \vec{\beta}}.$
Moreover this map coincides with multiplication in the algebra $\mathcal{V}_{\tilde{C}, \vec{p}, \vec{q}}^{\dagger}(G).$  
\end{proposition}

\begin{proof}
This all follows from Proposition \ref{factor} above and the observation that  $\chi_1 \times \chi_2$ and $\hat{\rho}_{\vec{\alpha} + \vec{\beta}}\circ C_{\vec{\lambda} ; \vec{\gamma}, \vec{\alpha} + \vec{\beta}, \vec{\alpha}^* + \vec{\beta}^*}^*(\chi_1 \otimes \chi_2)$ always have the same filtration level.
\end{proof}

Proposition \ref{filterblocks} shows that if the coweighting $(\Gamma, \vec{\theta})$ is generic, the associated graded algebra $gr_{\mathcal{F}_{\Gamma, \vec{\theta}}}(\mathcal{V}_{C, \vec{p}}^{\dagger})$ is isomorphic to the subalgebra of  $\mathcal{V}_{\tilde{C}, \vec{p}, \vec{q}}^{\dagger}(G)$ formed by the conformal block spaces $\mathcal{V}_{\tilde{C}, \vec{p}, \vec{q}}^{\dagger}(\vec{\lambda}, \vec{\alpha}, \vec{\alpha}^*, L)$; namely those with dual dominant weights assigned to paired points $q_1, q_2$ , and all levels $L$ on different connected components of $\tilde{C}$ equal.   This subalgebra can be realized as the algebra of invariants with respect to a certain torus.   For each oriented edge $e \in E(\Gamma)$ (respectively, ordered pair of points $q_1, q_2 \in \tilde{C}$ associated to a normalized singularity) there is an action of a maximal torus $T \subset G$ on  $\mathcal{V}_{\tilde{C}, \vec{p},  \vec{q}}^{\dagger}(G)$ whose character on the subspace $\mathcal{V}_{\tilde{C}, \vec{p}, \vec{q}}^{\dagger}(\vec{\lambda}, \vec{\alpha}, L)$ is the difference $\alpha_1 - \alpha_2^*$, where $\alpha_i$ is the label of $q_i.$  Similarly, if $\tilde{C}$ has multiple connected components, there action of $\C^*$ for each edge $e \in E(\Gamma)$ whose character is the difference between the levels on the spaces of conformal blocks associated to the components connected by $e.$ We let $T_{\Gamma}$ be the product of all such $T\times \C^*$.  The algebra of invariants $(\mathcal{V}_{\tilde{C}, \vec{p}, \vec{q}}^{\dagger}(G))^{T_{\Gamma}}$ is then the required sum of spaces of conformal blocks. 

Next we show that the associated graded algebra $gr_{\mathcal{F}_{\Gamma, \vec{\theta}}}(\mathcal{V}_{C, \vec{p}}^{\dagger})$ can be realized as a flat degeneration of $\mathcal{V}_{C, \vec{p}}^{\dagger}$.  We may form the Reese algebra $R_{(\Gamma, \vec{\theta})}[\mathcal{V}_{C, \vec{p}}^{\dagger}(G)] \subset \mathcal{V}_{C, \vec{p}}^{\dagger}(G)[t],$ defined as

\begin{equation}
R_{(\Gamma, \vec{\theta})}[\mathcal{V}_{C, \vec{p}}^{\dagger}(G)]  = \bigoplus_{N \geq 0} \mathcal{F}^{\leq N}_{\Gamma, \vec{ \theta}}[\mathcal{V}_{C, \vec{p}}^{\dagger}(G)].\\
\end{equation}

\noindent
The Reese algebra of a filtration is a standard construction in commutative algebra, it is naturally a $\C[t]$ algebra, where $t$ acts by sending an element in filtration level $N$ to its identical copy in level $N +1.$  The following properties are standard, see e.g. \cite{AB}: 

\begin{enumerate}
\item $R_{(\Gamma, \vec{\theta})}[\mathcal{V}_{C, \vec{p}}^{\dagger}(G)]$ is flat over $\C[t]$,\\
\item $\frac{1}{t} R_{(\Gamma, \vec{\theta})}[\mathcal{V}_{C, \vec{p}}^{\dagger}(G)] \cong \mathcal{V}_{C, \vec{p}}^{\dagger}(G)[t, \frac{1}{t}]$,\\
\item $R_{(\Gamma, \vec{\theta})}[\mathcal{V}_{C, \vec{p}}^{\dagger}(G)]/t$ is the associated graded algebra $(\mathcal{V}_{\tilde{C}, \vec{p}, \vec{q}}^{\dagger}(G))^{T_{\Gamma}}.$\\
\end{enumerate}

This construction gives a flat degeneration with generic fiber $\mathcal{V}_{C, \vec{p}}^{\dagger}(G),$ and special fiber equal to the subalgebra $(\mathcal{V}_{\tilde{C}, \vec{p}, \vec{q}}^{\dagger}(G))^{T_{\Gamma}}$ of $\mathcal{V}_{\tilde{C}, \vec{p},  \vec{q}}^{\dagger}(G).$  This proves Theorem \ref{T2} and Theorem \ref{main}. 

\begin{remark}
The coordinate ring of the group $G$ has the following decomposition into isotypical $G\times G$ components (see \cite[Theorem 12.9]{Gr}): 

\begin{equation}
\C[G] = \bigoplus_{\lambda \in \Delta} V(\lambda, \lambda^*).\\
\end{equation}

\noindent
The maps $\phi_{\eta, \eta^*}$ are the components of the dual of the multiplication map $m: \C[G] \otimes \C[G] \to \C[G]$
on this coordinate ring. Using Proposition \ref{reform} and a suitable modification of Diagram \ref{D}, we can construct an isomorphism of algebras out of the factorization maps $\rho_{\alpha}: \mathcal{H}(0, L)\otimes V(\vec{\lambda}) \to \mathcal{H}(0, L)\otimes V(\vec{\lambda}, \alpha, \alpha^*).$  

\begin{equation}
\rho: \mathcal{H}(0, L) \otimes V(\vec{\lambda}) \to \bigoplus_{\alpha \in \Delta_L} \mathcal{H}(0, L)\otimes V(\vec{\lambda},\alpha, \alpha^*)\\ 
\end{equation}

\begin{equation}
\rho^*: [\C[Q]\otimes \C[G/U]^{\otimes n} \otimes \C[G]]^{\hat{\mathfrak{g}}[\tilde{C}, p, q_1, q_2]} \cong [\C[Q] \otimes \C[G/U]^{\otimes n}]^{\hat{\mathfrak{g}}[C, p]} = \mathcal{V}_{C, \vec{p}}^{\dagger}(G)\\
\end{equation}

\noindent
Here $\C[Q] = \bigoplus_{L \geq 0} \mathcal{H}(0, L)^*$ is the total coordinate ring of the affine Grassmannian variety,
and $\C[G/U] = \bigoplus_{\lambda \in \Delta} V(\lambda)$ is the coordinate ring of the quotient of $G$ by a maximal unipotent subgroup. Let $G/\!\!/U = Spec(\C[G/U])$ be the GIT quotient of $G$ by $U$, notice that the grading by dominant weight corresponds to a residual $T$ action on $G/\!\!/U$. The degeneration constructed in this section is then induced by the so-called horospherical contraction of the group scheme $G$, see \cite{Po}. This is a flat $G\times G$ degeneration of $G$ to the scheme $[G/\!\!/U \times G/\!\!/U]/T$, where the $T$ action is defined through the residual $T\times T$ action on $ G/\!\!/U \times G/\!\!/U$ so that $V(\lambda)\otimes V(\lambda^*) \subset \C[G/\!\!/U]\otimes C[G/\!\!/U]$ is invariant. 
\end{remark}

\begin{remark}
The association of a filtration (actually a valuation) on $\mathcal{V}_{C, \vec{p}}^{\dagger}(G)$ to a coweighting $(\Gamma, \vec{\theta})$ suggests that coweightings should play a role in the tropical theory of the moduli of principal bundles, see \cite{Pa}. 
\end{remark}

\begin{example}\label{ex}

We compute an example for the curve in Figure \ref{fig:EX1}.

\begin{figure}[htbp]
\centering
\includegraphics[scale = 0.5]{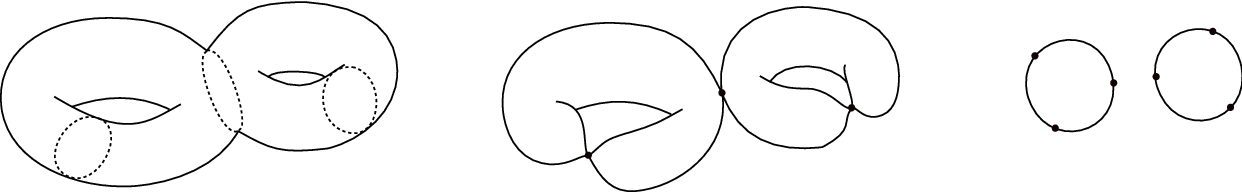}
\caption{Curve of genus $2,$ stable curve of genus $2$ and disjoint union
of two triple marked curves of genus $0.$}
\label{fig:EX1}
\end{figure}

Explicitly we have,

\begin{equation}
(\mathcal{V}_{\tilde{C'}, \vec{q}}^{\dagger}(G))^T = (\mathcal{V}_{\mathbb{P}^1, q_1, q_2, q_3}^{\dagger}(G) \otimes \mathcal{V}_{\mathbb{P}^1, q_4, q_5, q_6}^{\dagger}(G))^T = \\
\end{equation}

$$ \bigoplus_{L, \alpha, \beta, \gamma} [\mathcal{V}^{\dagger}_{\mathbb{P}^1, q_1, q_2, q_3}(\alpha, \alpha^*, \beta, L) \otimes \mathcal{V}^{\dagger}_{\mathbb{P}^1, q_4, q_5, q_6}(\beta^*, \gamma, \gamma^*, L)].$$

Multiplication is computed component-wise over the tensor product.
In the case $\mathfrak{g} = sl_2(\C),$ this ring is the semigroup of weightings on the
graph pictured in Figure \ref{fig:EX2}, where the middle edge is always weighted even and less than or equal to twice either of the loop edges, and the sum of twice either loop edge and the middle edge is bounded by the level.

\begin{figure}[htbp]
\centering
\includegraphics[scale = 0.3]{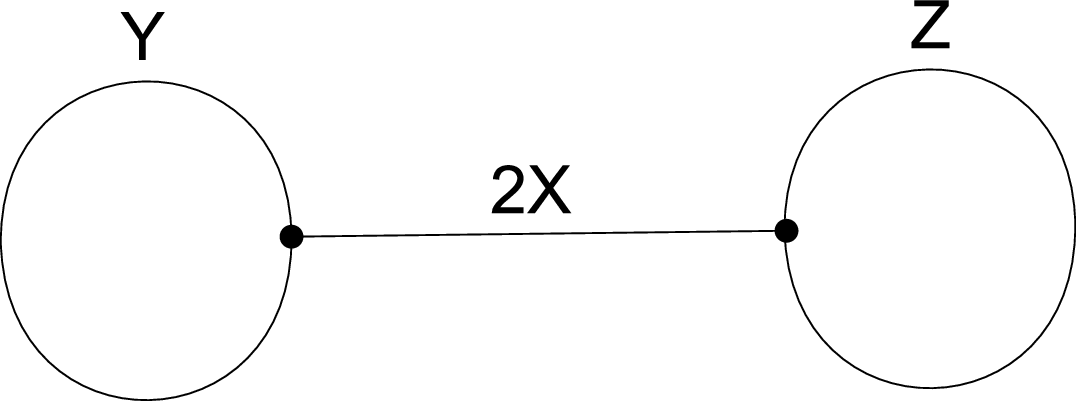}
\caption{Weighted genus $2$ graph.}
\label{fig:EX2}
\end{figure}

\end{example}

\section{Correlation and the genus 0 case}\label{branch}

In this section we prove Theorem \ref{G0}, which relates the degeneration on $\mathcal{V}^{\dagger}_{C, \vec{p}}(G)$ to a similar construction on an algebra $\hat{\mathfrak{A}}_n$ which we now describe.  Recall that $\mathfrak{A}_n$ is the algebra of left diagonal $G$ invariants in $\C[G/U]^{\otimes n}$, we let $M_n = Spec(\mathfrak{A}_n)$.  The coordinate ring of $G/U$ is known to be a multiplicity-free direct sum of the irreducible representations of $G$  (see e.g. \cite{PV}) with multiplication computed by the maps $C_{\lambda ; \gamma}^*: V(\lambda^*) \otimes V(\gamma^*) \to V(\lambda^* + \gamma^*)$.   It follows that the coordinate ring $\mathfrak{A}_n$ of this space is the graded direct sum $\bigoplus_{\vec{\lambda} \in \Delta^n} [V(\vec{\lambda})]^{\mathfrak{g}}$.  The algebra $\hat{\mathfrak{A}}_n$ is defined to be the following Rees algebra of $\mathfrak{A}_n$:

\begin{equation}
\hat{\mathfrak{A}}_n = \bigoplus_{\vec{\lambda} \in \Delta^n, \lambda_i(\theta^{\vee}) \leq L} [V(\vec{\lambda})]^{\mathfrak{g}}t^L.\\
\end{equation}

Now we may observe that there is a correlation map $F_{C, \vec{p}} :\mathcal{V}^{\dagger}_{C, \vec{p}}(G) \to \hat{\mathfrak{A}}_n$ which is a map of algebras by Lemma \ref{mainsquare}.

\subsection{The correlation morphism}

The invariant spaces $[V(\vec{\lambda})]^{\mathfrak{g}}$ also have a factorization property, so a simplified version of the construction in Section \ref{filter} applies the algebra $\mathfrak{A}_n$ (see also \cite{M2}).

\begin{proposition}\label{F}
We let $S_1 \cup S_2 = [n]$ be a partition of $n$ indices, and we let $\tree$ be the tree with leaves in bijection with $[n]$ and one interior edge defined by this partition.   There is a factorization isomorphism $\sum \rho_{\eta}: \bigoplus_{\eta} [V(\vec{\lambda}_1, \eta)]^{\mathfrak{g}} \otimes [V( \eta^*, \vec{\lambda}_2)]^{\mathfrak{g}} \to  [V(\vec{\lambda}_1, \vec{\lambda}_2)]^{\mathfrak{g}}$.  This defines a filtration $\mathcal{F}_{\tree, \vec{\theta}}$ on $\hat{\mathfrak{A}}_n$ for each coweighting of the tree $\tree$.  If the components of $\vec{\theta}$ are strictly positive on all positive roots, then the associated graded algebra is isomorphic to $[\hat{\mathfrak{A}}_{|S_1|}\otimes \hat{\mathfrak{A}}_{|S_2|}]^{T\times \C^*}.$
\end{proposition}

\begin{proof}
Apply the steps of Section \ref{filter}.  
\end{proof}

   In Section \ref{sheaf} we showed that there is a $1-1$ map $F_{C, \vec{p}}: \mathcal{V}^{\dagger}_{C, \vec{p}}(G) \to \hat{\mathfrak{A}}_n$ when $C$ is a genus $0$ stable curve.  Next we show that over a stable curve of type $\tree,$ the correlation morphism $F_{C, \vec{p}}: \mathcal{V}^{\dagger}_{C, \vec{p}}(G) \to \hat{\mathfrak{A}}_n$ intertwines the factorization of conformal blocks with the branching decomposition defined by $\tree$ on $\hat{\mathfrak{A}}_n.$     Let $C, \vec{p}_1, \vec{p}_2$ be a stable curve of type $\tree,$ with $\tilde{C} = C_1 \cup C_2.$  We have a commuting square of $\mathfrak{g}$ representations, 

$$
\begin{CD}
\mathcal{H}(\vec{\lambda}, L) @> \rho_{\eta} >> \mathcal{H}(\vec{\lambda_1}, \eta, L) \otimes \mathcal{H}(\eta^*, \vec{\lambda_2}, L) \\
@Ai AA @A i \otimes i AA\\
V(\vec{\lambda}) @> \rho_{\eta} >> V(\vec{\lambda_1}, \eta) \otimes V(\eta^*, \vec{\lambda_2})\\
\end{CD}
$$

\noindent
We may dualize this diagram, and take invariants with respect to the Lie algebras $\hat{\mathfrak{g}}[C, \vec{p}_1, \vec{p}_2]$ on the top left, $\hat{\mathfrak{g}}[C_1 \setminus \vec{p}_1] \oplus \hat{\mathfrak{g}}[C_2 \setminus \vec{p}_2]$ on the top right, $\mathfrak{g}$ on the bottom left, and $\mathfrak{g} \oplus \mathfrak{g}$ on the bottom right.    The top morphism $\hat{\rho}_{\eta}: [\mathcal{H}(\vec{\lambda_1}, \eta, L) \otimes \mathcal{H}(\eta^*, \vec{\lambda_2}, L)]^* \to \mathcal{H}(\vec{\lambda}, L)^*$ on the dual spaces takes the graded component $\mathcal{V}_{C_1, \vec{p}_1, q_1}^{\dagger}(\vec{\lambda}_1, \eta, L) \otimes \mathcal{V}_{C_2, \vec{p}_2, q_2}^{\dagger}(\eta^*, \vec{\lambda}_2, L)$ into $\mathcal{V}_{C, \vec{p}_1, \vec{p}_2}^{\dagger}(\vec{\lambda}, L)$ by the factorization rules.  The bottom morphism $\hat{\rho}_{\eta}: [V(\vec{\lambda_1}, \eta) \otimes V(\eta^*, \vec{\lambda_2})^*]^{\mathfrak{g}} \to V(\vec{\lambda})^*$ takes  $[V(\vec{\lambda}_1 ,\eta)] \otimes [V(\eta^*,\vec{\lambda}_2)]^{\mathfrak{g}}$ to $[V(\vec{\lambda})]^{\mathfrak{g}}$ by Proposition \ref{F}. The fact that these invariant spaces are connected by the appropriate morphisms, along with the commutivity of the dual diagram implies that the following diagram commutes.

$$
\begin{CD}
\mathcal{V}_{C, \vec{p}_1, \vec{p}_2}^{\dagger}(\vec{\lambda}, L) @< \hat{\rho}_{\eta} << \mathcal{V}_{C_1, \vec{p}_1, q_1}^{\dagger}(\vec{\lambda}_1, \eta, L) \otimes \mathcal{V}_{C_2, \vec{p}_2, q_2}^{\dagger}(\eta^*, \vec{\lambda}_2, L) \\
@V F_{C, \vec{p}_1, \vec{p}_2} VV @V F_{C_1, \vec{p}_1, q_1} \otimes F_{C_2, q_2 \vec{p}_2} VV \\
[V(\vec{\lambda})^*]^{\mathfrak{g}} @< \hat{\rho}_{\alpha} << [V(\vec{\lambda}_1 ,\eta)^*]^{\mathfrak{g}} \otimes[V(\eta^*,\vec{\lambda}_2)^*]^{\mathfrak{g}}\\
\end{CD}
$$

\noindent
This shows that the direct sum decompositions of $\mathcal{V}_{C, \vec{p}_1, \vec{p}_2}^{\dagger}(\vec{\lambda}, L)$ and $[V(\vec{\lambda})]^{\mathfrak{g}}$ from the factorization rules are compatible, and implies that the the filtrations on the branching algebras and the algebras of conformal blocks agree. Theorem \ref{G0} follows by induction.   

\begin{remark}
Whenever $\tilde{C}$ is a disjoint union $C_1 \cup C_2$ 
the diagram above commutes.  This implies a version
of theorem \ref{G0} is true for general genus, except the correlation
$F_{C, \vec{p}}$ is no longer a monomorphism. 
\end{remark}

\subsection{The case $0, 3$}

For a genus $0,$ triple marked curve there is no moduli, $\mathcal{M}_{0, 3} = \{pt\},$ so the algebra of conformal blocks is unique.  In this case, conformal blocks have a purely representation-theoretic description as a subspace of the space of invariants.  See \cite{TUY} for the following. 

\begin{proposition}\label{tensblock}
The space $\mathcal{V}_{0,3}^{\dagger}(\lambda, \gamma, \mu, L) \subset [V(\lambda^*)\otimes V(\gamma^*) \otimes V(\mu^*)]^{\mathfrak{g}}$  has the following description.  Consider the factorization of $V(\lambda^*)$, $V(\gamma^*)$ and $V(\mu^*)$ as $sl_2(\C)$ representations with respect to the longest root $\theta: sl_2(\C) \to \mathfrak{g},$ $V(\lambda^*) = \bigoplus W(\lambda^*, i) \otimes V(i),$ $V(\gamma^*) = \bigoplus W(\gamma^*, j) \otimes V(j),$ $V(\mu^*) = \bigoplus W(\mu^*, k) \otimes V(k).$
Let $W(\lambda^*, \gamma^*, \mu^*, L)$ be the subspace of $V(\lambda^*)\otimes V(\gamma^*) \otimes V(\mu^*)$ of components $V(i)\otimes V(j) \otimes V(k)$ with $i + j + k \leq 2L$, then:  

\begin{equation}
\mathcal{V}_{0, 3}^{\dagger}(\lambda, \gamma, \mu, L) = W(\lambda^*, \gamma^*, \mu^*, L) \cap (V(\lambda^*)\otimes V(\gamma^*) \otimes V(\mu^*))^{\mathfrak{g}}.\\
\end{equation}

\end{proposition}

\begin{remark}\label{canonicaldegen}
The degeneration constructed in Section \ref{filter} could be completed to a toric degeneration given any $T^3$-invariant toric degeneration of $\mathcal{V}_{0, 3}^{\dagger}(G)$.   A sufficient condition for the existence of such a filtration would be a basis $B(\lambda, \gamma, \mu)$ of each space $[V(\lambda)\otimes V(\gamma) \otimes V(\mu)]^{\mathfrak{g}}$ with the following properties: 

\begin{enumerate}
\item The bases $B(\lambda, \gamma, \mu)$ have a ``lower-triangular multiplication" property with respect to the multiplication in $\mathfrak{A}_3$,\\
\item The intersection $B(\lambda, \gamma, \mu) \cap \mathcal{V}_{0, 3}^{\dagger}(\lambda, \gamma, \mu, L) \subset [V(\lambda^*) \otimes V(\gamma^*) \otimes V(\mu^*)]^{\mathfrak{g}}$ is a basis for each $L$.\\
\end{enumerate}

Lusztig's dual canonical basis satisfies the first property above, and the resulting degenerations are explored in \cite{M2}.  In \cite{M4} we use this technique to build toric degenerations of the algebra of conformal blocks when $G = SL_3(\C).$
\end{remark}

\subsection{Projective coordinate rings}\label{cspaces}

Recall that the coordinate ring $\mathfrak{A}_n$ is the multiplicity-free direct sum of the invariant spaces $V_{\mathfrak{g}}^{\dagger}(V(\vec{\lambda}))$, and that this algebra is graded by the tuples of dominant weights $\vec{\lambda} \in \Delta^n.$  This grading corresponds to a right action of $T^n$ on $M_n$.  We let $M_{\vec{\lambda}}$ be the GIT qoutient of $M_n$ by $T^n$ with respect to the character defined by the tuple $\vec{\lambda}$; this is the space of configurations associated to $\vec{\lambda}.$  The projective coordinate ring naturally associated to $M_{\vec{\lambda}}$ by this construction is then the direct sum $\mathfrak{A}_{\vec{\lambda}} \subset \mathfrak{A}_n$ of the invariant spaces  $[V(K\vec{\lambda})]^{\mathfrak{g}}$ for $K \geq 0.$ 

All of the techniques we have used to study algebras of conformal blocks and branching algebras are carried out on graded pieces of
these algebras.  Because of this, much of what we say can be extended to nice graded subalgebras, in particular these statements apply to the projective coordiante ring of the coarse moduli spaces $\mathcal{M}_{C, \vec{p}}(\vec{\lambda}, L)$.  In particular, the correlation map $F_{C, \vec{p}}: \mathcal{V}^{\dagger}_{C, \vec{p}}(G) \to \hat{\mathfrak{A}}_n$ induces a map on projective coordinate rings:

\begin{equation}F_{C, \vec{p}}^{\vec{\lambda}, L}: R_{C, \vec{p}}(\vec{\lambda}, L) \to \mathfrak{A}_{\vec{\lambda}}.\end{equation}

\noindent
When $g = 0,$ this map is a monomorphism, in which case can deduce the following.

\begin{proposition}
For $g = 0$ and $L >> 0$ the map $F_{C, \vec{p}}^{\vec{\lambda}, L}$ above
is an isomorphism. 
\end{proposition}

\begin{proof}
The algebra $\mathfrak{A}_{\vec{\lambda}}$ is finitely generated, say by the spaces $[V(\vec{N_i\lambda^*})]^{\mathfrak{g}}$ for $i = 1, \ldots, k$. 
Each of these spaces is filtered by the spaces of conformal blocks $\mathcal{V}_{C, \vec{p}}^{\dagger}(\vec{\lambda}_i, L)$. It follows that if $L$ is chosen to be sufficiently large,  $\mathcal{V}_{C, \vec{p}}^{\dagger}(\vec{\lambda}_i, L) = [V(N_i\vec{\lambda^*})]^{\mathfrak{g}}$, so 
that $ R_{C, \vec{p}}(\vec{\lambda}, L) \cong \mathfrak{A}_{\vec{\lambda}}.$
\end{proof}

Compare this proposition to Remark $4.3$ in \cite{TW}.  The algebra $R_{C, \vec{p}}(\vec{\lambda}, L)$ is the projective coordinate ring of $M_{C, \vec{p}}^{ss}(\vec{\lambda}, L),$ the coarse moduli space of semi-stable bundles, where the semi-stability condition is determined by the data $(\vec{\lambda}, L).$ Degenerations associated to labelled trees carry over to these algebras as well.  This implies that for large $L,$ a toric degeneration of $M_{\vec{\lambda}}$ gives a toric degeneration of a ring of generalized theta functions.  Toric degenerations of $\mathfrak{A}_{\vec{\lambda}}$ can be constructed from a toric degeneration of $\mathfrak{A}_n$, see \cite{M2} and \cite{M14}.

\section{The case $\mathfrak{g} = sl_2(\C)$}\label{sl2}

The results of this section should be of independent interest for readers interested in the $\Z/2\Z$ group-based phylogenetic statistical models, \cite{BBKM}, \cite{Bu}, \cite{BW}. We refer the reader to \cite{KM} for the connection between conformal blocks and other group-based  models.

\subsection{Proof of Theorem \ref{degensl2}}

 For $sl_2(\C)$ dominant weights $r_1, r_2, r_3 \in \Z_{\geq 0}$, the spaces  $(V(r_1)\otimes V(r_2)\otimes V(r_3))^{sl_2(\C)}$ are multiplicity free, and are non-trivial when $r_1 + r_2 + r_3 \in 2\Z$ and $|r_1 - r_3| \leq r_2 \leq r_1 + r_3,$ these are known as the Clebsch-Gordon conditions.  Proposition \ref{tensblock} implies that the spaces $\mathcal{V}_{0,3}^{\dagger}(r_1, r_2, r_3, L)$ are also multiplicity free, and that they are non-trivial when $r_1 + r_2 + r_3 \leq 2L$ and the Clebsch-Gordon conditions are satisfied, these are known as the Quantum Clebsch-Gordon conditions.  The following is a consequence of the fact that the weights $(r_1, r_2, r_3, L)$ define a multigrading of $\mathcal{V}_{0,3}^{\dagger}(SL_2(\C))$.

\begin{proposition}
For $G = SL_2(\C)$ the algebra $\mathcal{V}_{0, 3}^{\dagger}(SL_2(\C))$ is isomorphic to the graded affine semigroup algebra associated to the polytope $P_3 = conv\{(0,0,0),$ $(1, 1, 0),$ $(1, 0, 1),$ $(0, 1, 1)\} \subset \R^3,$ with the respect to the lattice determined by $r_i \in \Z$, $r_1 + r_2 + r_3 \in 2\Z$.
\end{proposition}

Now we describe the algebra $\bigotimes_{v \in V(\Gamma)} \mathcal{V}_{0, 3}^{\dagger}(SL_2(\C)) = \C[ P_3^{\times V(\Gamma)}]$ and its $T_{\Gamma}$-invariant subalgebra.  We have associated a copy of $P_3$ to each vertex $v \in V(\Gamma)$, and likewise each entry $r_1, r_2, r_3$ of a point in this $P_3$ is assigned to an edge incident on $v.$  The isotypical spaces of the $T_{\Gamma}$ action on $\C[ P_3^{\times V(\Gamma)}]$ are each spanned by a lattice point $\vec{w} \in P_3^{V(\Gamma)}$. The character of the $T\times \C^*$ action associated to a given edge $e \in E(\Gamma)$ on a $\vec{w}$ returns the difference $L_v(\vec{w}) - L_u(\vec{w})$ of the levels on the end points $\{u, v\}$ of $e$, and the difference $w_v(e) - w_u(e)$ between the weights assigned $e$ by the $v$ and $u$ components of $\vec{w}.$   Taking the torus invariants $[\C[P_3]^{\otimes |V(\Gamma)|}]^{T_{\Gamma}}$ therefore picks out exactly those lattice points with components from the same Minkowski sum $L\circ P_3,$ which consistently weight the edges of $\Gamma,$ see Figure \ref{fig:identification}. By definition, the invariant subalgebra is the affine semigroup algebra $\C[P_{\Gamma}]$, this proves Theorem \ref{degensl2}.

\begin{figure}[htbp]
\centering
\includegraphics[scale = 0.4]{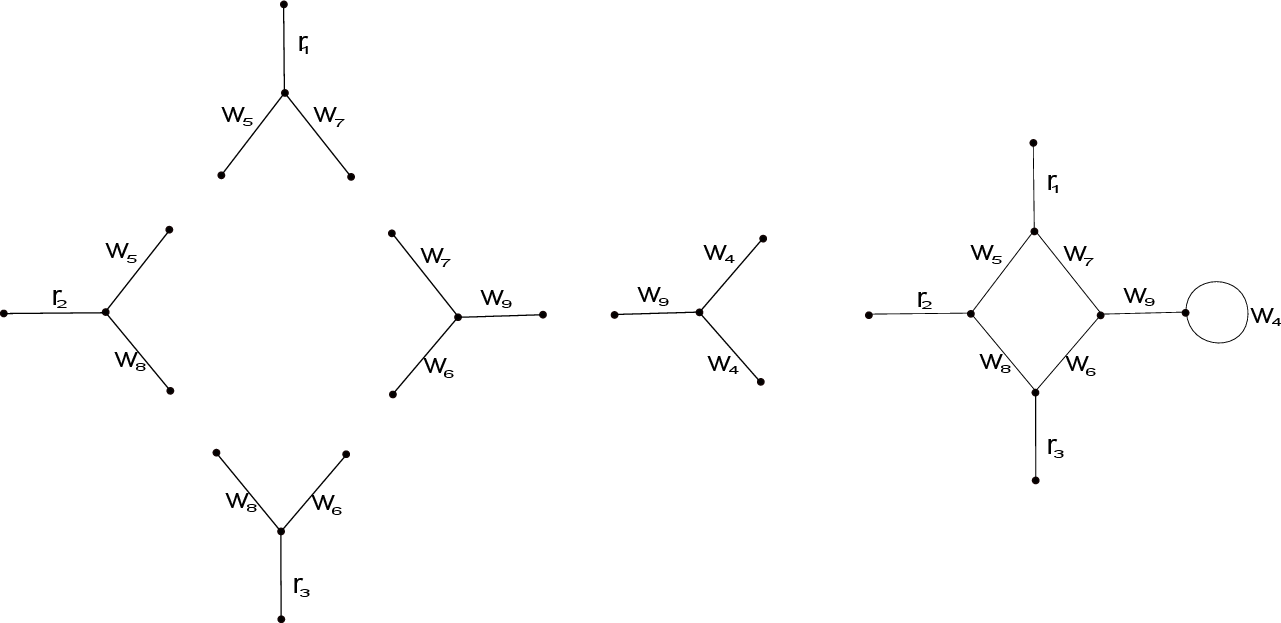}
\caption{A $T_{\Gamma}$ invariant.}
\label{fig:identification}
\end{figure}

\subsection{$P_{\Gamma}$ is a Newton-Okounkov body}\label{NOKbody}

As an application of Propositions \ref{factor} and \ref{filterblocks} we interpret the polytope $P_{\Gamma}$ as a Newton-Okounkov body for the algebra of conformal blocks over the marked curve of type $\Gamma$.  The results of this subsection will not be needed elsewhere in the paper.  For background on the theory of Newton-Okounkov bodies we refer the reader to the papers of Kaveh, Khovanskii \cite{KK} and Lazarsfeld, Musta{\c{t}}{\u{a}} \cite{LM}.   A Newton-Okounkov body $P$ of a graded algebra $A = \bigoplus_{L \geq 0} A_L$  is a combinatorial invariant associated to a choice of a rank $a+1$ valuation $v: A \to \Z^{a+1}$  (see \cite[Definition 2.8]{KK}) which refines the grading on $A$. Here $a+1$ is the Krull dimension of $A$ and $\Z^{a+1}$ is given a lexicographic ordering which prioritizes the grading component. In particular, we assume that $v(f) = (\ldots, deg(f)) \in \Z^{a+1}$ for $f$ homogeneous.  The image $v(A) \subset \Z^{a+1}$ is an affine semigroup (see \cite[Proposition 2.10]{KK}), and its convex hull $\bar{P} \subset \R^{a+1}$ is a convex cone.   The slice $P = \bar{P}\cap \R^a\times \{1\} \subset \R^{a+1}$ is called the Newton-Okounkov body of $A$ associated to $v.$

Now we fix $A = \mathcal{V}^{\dagger}_{C, \vec{p}}(SL_2(\C))$, where $(C, \vec{p})$ is the stable curve of type $\Gamma.$  The following Proposition is a direct consequence of Propositions \ref{factor} and  \ref{filterblocks}.

\begin{proposition}\label{porder}
As a vector space, the algebra $\mathcal{V}^{\dagger}_{C, \vec{p}}(SL_2(\C))$ is a direct sum of one dimensional spaces $\C\phi_{w,L}$, where $w$ is a lattice point in the $L$-th Minkowski sum $L\circ P_{\Gamma}.$  The product $\phi_{w, L}\phi_{w', K}$ is a sum of the form $\phi_{w + w', L+K} + \sum_{u < w + w'} C_u\phi_{u,L+K}$, where the partial ordering $u < w + w'$ means that $u(e) \leq w(e) + w'(e)$ for each edge $e \in E(\Gamma)$ and there is some edge where this inequality is strict. 
\end{proposition}

By placing a total ordering $\ll$ on the edge set $E(\Gamma)$ we can produce a lexicographic ordering $\prec$ on the basis members $\phi_{w, L}$ as follows.  We say $\phi_{w, L} \prec \phi_{w', K}$ if $(w(e_1), \ldots, w(e_{|E(\Gamma)|}), L)$ is less than $(w'(e_1), \ldots, w(e_{|E(\Gamma)|}), K)$ in the lexicographic ordering, where the set $E(\Gamma)$ has been placed in bijection with the integers $\{1, \ldots, |E(\Gamma)|\}$ using $\ll$.  Using $\prec,$ we define a filtration on $\mathcal{V}^{\dagger}_{C, \vec{p}}(SL_2(\C))$ by the vector spaces $F_{\Gamma}^{\preceq w} = \bigoplus_{w' \preceq w} \C\phi_{w'}$. Proposition \ref{porder} (along with Propositions \ref{factor} and \ref{filterblocks}) implies the following proposition. 

\begin{proposition}
The filtration $F_{\Gamma}$ is an algebra filtration of $\mathcal{V}_{C, \vec{p}}^{\dagger}(SL_2(\C))$.  The associated graded algebra of the filtration $F_{\Gamma}$ is isomorphic to the affine semigroup algebra $\C[P_{\Gamma}].$ 
\end{proposition}

As $\C[P_{\Gamma}]$ is a domain, the filtration $F_{\Gamma}$ can be used to define a valuation $v_{\Gamma}$, see \cite[Proposition 3.2]{M14}:

\begin{equation}
f \in \mathcal{V}_{C, \vec{p}}^{\dagger}(SL_2(\C)) \ \ \ \ v_{\Gamma}(f) = Min\{ w | f \in F_{\Gamma}^{\preceq w}\}.\\
\end{equation}

\begin{proposition}\label{NOK}
The Newton-Okounkov body of $\mathcal{V}_{C, \vec{p}}^{\dagger}(SL_2(\C))$ with respect to the valuation $v_{\Gamma}$ is the polytope $P_{\Gamma}.$
\end{proposition}

\begin{proof}
The image of $v_{\Gamma}$ is the set of all $(w, L)$ for $w$ a lattice point in $L \circ P_{\Gamma}.$  It follows that the level $1$ slice can be identified with the closure of the set of all $\frac{1}{L}w$; this is $P_{\Gamma}$ itself. 
\end{proof}

\begin{remark}
An identical construction to the one given above shows that the polytope $P_{\Gamma}(\vec{r}, L)$ studied in \cite{M} and \cite{M3} is a Newton-Okounkov bodies for $R_{C, \vec{p}}(\vec{r}, L)$, where $(C, \vec{p})$ is the stable curve of type $\Gamma$. 
\end{remark}

For each connected, trivalent graph $\Gamma$ with $\beta_1(\Gamma) = g$ and $n$ leaves, Theorem \ref{degensl2} produces a tool to study $\mathcal{V}_{C, \vec{p}}^{\dagger}(SL_2(\C))$ for generic $(C, \vec{p})$. We prove Theorem \ref{sl2} by focusing on the algebra $\C[P_{\Gamma_{g, n}}]$ attached to a particular graph $\Gamma_{g, n},$ see Figure \ref{fig:virus}.

\begin{figure}[htbp]
\centering
\includegraphics[scale = 0.35]{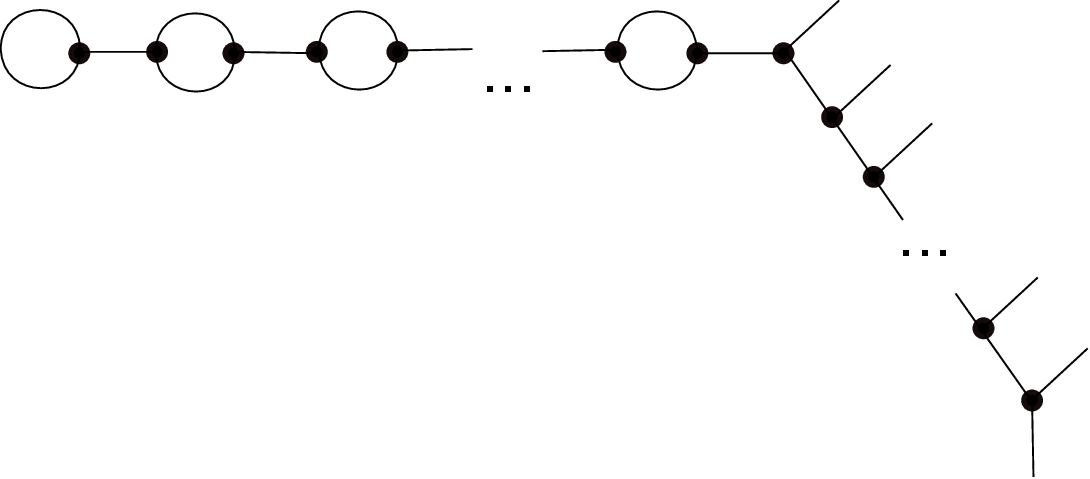}
\caption{The graph $\Gamma_{g, n}$}
\label{fig:virus}
\end{figure}

\subsection{Generators}

The following three lemmas show that $P_{\Gamma_{g, n}}$ and the second Minkowski sum $2 \circ P_{\Gamma_{g, n}}$ suffice to generate $\C[P_{\Gamma_{g, n}}]$.

\begin{lemma}\label{g1}
The case $(g, n)$ follows from cases $(g, 1)$ and $(0, n + 1)$
\end{lemma}

\begin{lemma}\label{2L+1}
The multiplication map $\pi: P_{\Gamma_{g, 1}}\times 2L\circ P_{g, 1} \to (2L+1)\circ P_{\Gamma_{g, 1}}$ is surjective.
\end{lemma}

\begin{lemma}\label{2L}
The multiplication map $\pi: [2\circ P_{\Gamma_{g, 1}}]^{\times L} \to 2L\circ P_{\Gamma_{g, 1}}$ is surjective.
\end{lemma}

In \cite{BW} Buczy\'{n}ska and Wi\'{s}niewski show that $\C[P_{\Gamma_{0, n}}]$ is generated by the lattice points of $P_{\Gamma_{0, n}}$, and that the ideal of relations on these generators has a quadratic, square-free Gr\"obner basis. It follows that $P_{0, n}$ is generated by weightings which weight all edges $e \in E(\Gamma_{0, n})$ with $0$ or $1.$

\subsection{Proof of Lemma \ref{g1}}

We prove Lemma \ref{g1} with a toric fiber product argument. The weight on any edge of $\Gamma_{g, n}$ coming from a lattice point of $L\circ P_{\Gamma_{g, n}}$ must be less than or equal to $L$, this follows directly from the defining inequalities. The lattice under consideration forces each horizontal edge in $\Gamma_{g, 1}$ to be weighted with an even number, including the leaf edge. We consider an element of $L\circ P_{\Gamma_{g, n}}$ ($g > 0$) as an element of $L\circ P_{\Gamma_{g, 1}}$ glued to an element of $L\circ P_{\Gamma_{0, n+1}},$ see Figure \ref{Fig2}. 

\begin{figure}[htbp]
\centering
\includegraphics[scale = 0.5]{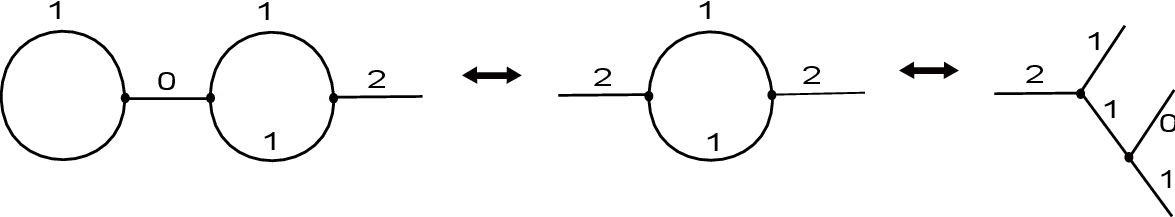}
\caption{}
\label{Fig2}
\end{figure}

Let $\omega \in L\circ P_{\Gamma_{g, n}},$  and let $\omega^g \in L\circ P_{\Gamma_{g, 1}}$ and $\omega^0 \in L\circ P_{\Gamma_{0, n+1}}$ be the restrictions of $\omega$ to the copies of $\Gamma_{g, 1}$ and $\Gamma_{0, n+1}$ in $\Gamma_{g, n}$ respectively.  By the theorem
of Buczy\'{n}ska and Wi\'{s}niewski, $\omega^0$ decomposes into $L$ weightings of $\Gamma_{0, n+1}$.

\begin{equation}
\omega^0 = \eta_1 + \ldots + \eta_L\\
\end{equation}

Now suppose that the generation statement of Theorem \ref{main} holds for $P_{\Gamma_{g, 1}},$ then 
$\omega^g$ likewise decomposes into elements of degree $1$ and $2$.

\begin{equation}
\omega^g = [\alpha_1 + \ldots + \alpha_k] + [\beta_1 + \ldots + \beta_m]\\
\end{equation}

\noindent
Here $2m + k = L$.  Due to the restrictions imposed by the lattice, the $\alpha_i$ all have weight $0$ on the edge shared by $\Gamma_{g, 1}$ and $\Gamma_{0, n+1}.$  We let $\eta_1, \ldots, \eta_{L'}$ and $\beta_1, \ldots, \beta_m'$ be the elements with a non-zero weight on this edge
in $P_{\Gamma_{0, n+1}}$ and $P_{\Gamma_{g, 1}}$ respectively. We must have $L' = 2m',$ so we may pair up the elements $\eta_1 + \eta_2, \ldots$ to make $m'$ elements in $2\circ P_{\Gamma_{0, n+1}},$ and we can glue these to the $\beta_i$ along the shared edge.  What remains must be exactly $k'$ $\eta_i'$s and $k'$ $\alpha_i'$s, both with $0$ along the shared edge, which can then likewise be glued along the edge in any order.

\subsection{Proof of Lemma \ref{2L+1}}

We prove Lemma \ref{2L+1} with an analysis of weightings on the component graphs in Figure \ref{Fig3}. We define the graded semigroups $B_1$ and $B_2$ accordingly, these semigroups are composed of those weightings on the pictured graphs which are even on the leaf edges. We let $o_1$ and $o_2$, be elements which weight the edges in a loop $1$ and all other edges $0$, in $B_1$ and $B_2$ respectively.  The following proposition is the crux of Lemmas \ref{2L+1} and \ref{2L}.

\begin{figure}[htbp]
\centering
\includegraphics[scale = 0.5]{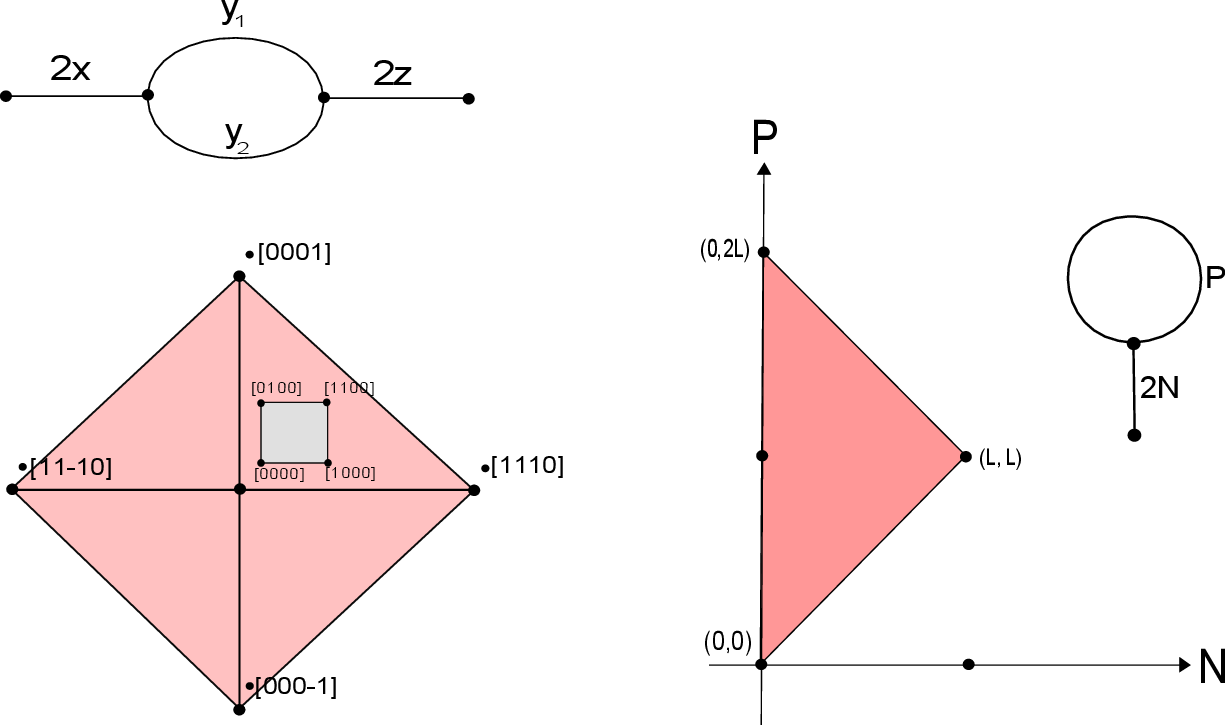}
\caption{Polytopes of weightings from $B_1(2L)$ and $B_2(2).$}
\label{Fig3}
\end{figure}

\begin{proposition}\label{bb}
The semigroups $B_1$ and $B_2$ are generated by elements of degrees $1,$ and $2$ (these are the elements appearing in  Figures \ref{Fig5} and \ref{Fig6} respectively).
\end{proposition}

\begin{proof}
As $B_1$ and $B_2$ are both the semigroup of lattice points in a convex cone, we directly verify that the elements of degrees $1$ and $2$ generate by checking that they constitute a \emph{Hilbert basis} using the software package Macaulay 2 \cite{mac2} or 4ti2 \cite{4ti2}. This software uses the project-and-lift algorithm, see \cite{Hemmecke}.
\end{proof}

\begin{figure}[htbp]
\centering
\includegraphics[scale = 0.35]{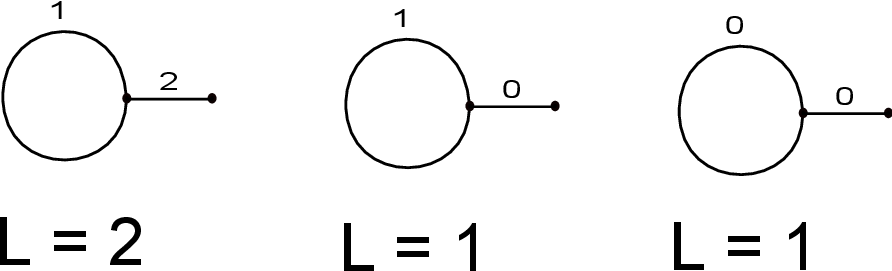}
\caption{}
\label{Fig5}
\end{figure}

\begin{figure}[htbp]
\centering
\includegraphics[scale = 0.35]{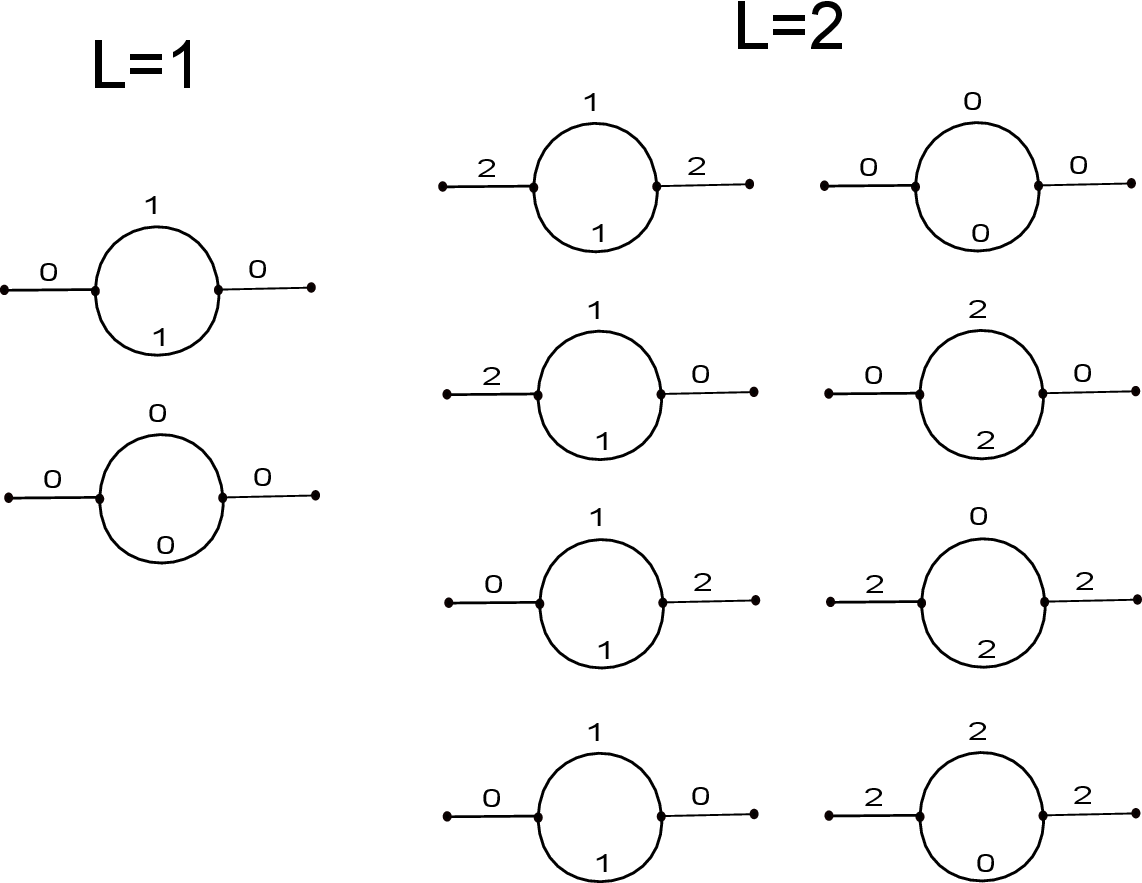}
\caption{}
\label{Fig6}
\end{figure}

It follows that all generators weight the leaf edges of either graph with either $0$ or $2.$ We prove Lemma \ref{2L+1} by pulling all copies of $o_1$ and $o_2$ off the loops in $\Gamma_{g, 1}.$ If a loop has a vertex with incident entries adding to $2L+1$ we can extract an $o_1$ or $o_2$ from that loop, otherwise we leave it alone.   The resulting collection of loops is a member of $P_{\Gamma_{g, n}},$ and the remainder is an element of $2L \circ P_{\Gamma_{g, n}},$ see Figure \ref{level1}.

\begin{figure}[htbp]
\centering
\includegraphics[scale = 0.35]{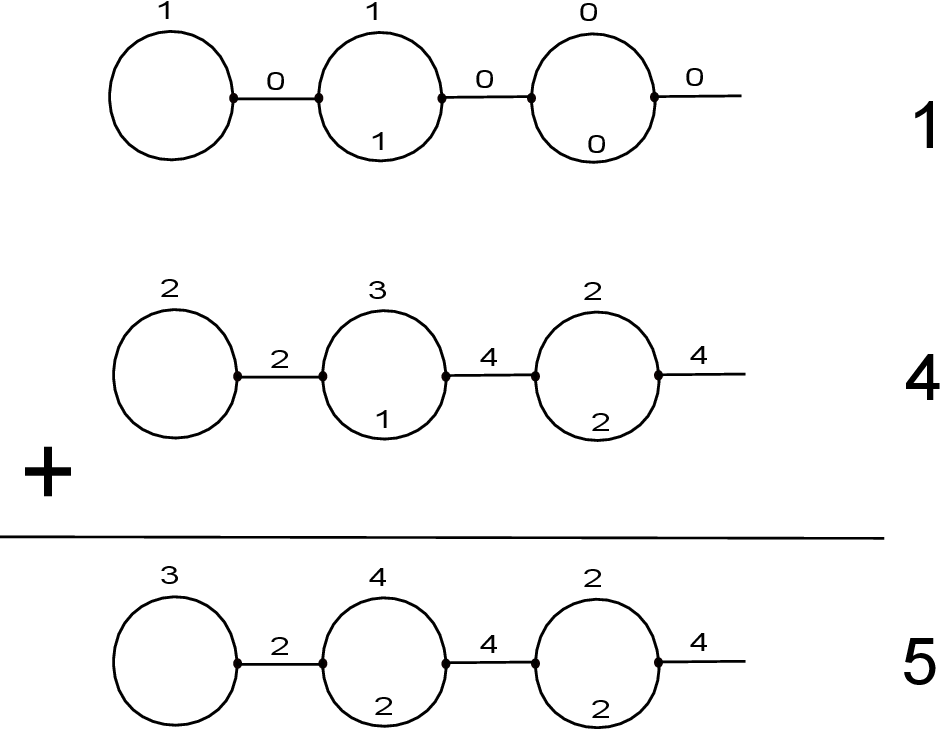}
\caption{}
\label{level1}
\end{figure}

\subsection{Proof of Lemma \ref{2L}}

A weighting of $\Gamma_{g, 1}$ automatically induces a weighting of $\Gamma_{g-1, 1}$ and an element of $B_2$.  Now we are in a position to use an argument similar to the proof of Lemma \ref{g1}. The theorem is true for $\C[P_{\Gamma_{1, 1}}] = \C[B_1]$, suppose that it holds for $\Gamma_{g -1, 1}.$  For an element $\omega \in 2L\circ P_{\Gamma_{g, 1}}$ we consider
the restrictions $\omega_1, \omega_2$ to $\Gamma_{g-1, 1}$ and $B_2$ respectively. 
By induction, $\omega_1$ and $\omega_2$ can be factored into $L$ elements of $2\circ P_{\Gamma_{g-1, 1}}$ and $B_2(2)$, respectively. 

\begin{equation}
\omega_1 = \alpha_1 + \ldots + \alpha_L, \ \ \omega_2 = \beta_1 + \ldots + \beta_L\\
\end{equation}

The number of $\alpha_i$ which weight the edge shared by $\Gamma_{g-1, 1}$ and $B_2$ with a $2$
must equal the number of $\beta_i$ which also weight this edge $2.$ Elements which weight this edge the same can therefore be matched up, this proves Lemma \ref{2L}.

\subsection{Relations}

We will describe relations for the building block semigroups $B_1, B_2$, then we show that these relation results
are stable as these building blocks are glued together.  Analysis of the building block semigroups gives the $(g, 1)$ case, which we combine with results in \cite{BW} to handle the $(g, n)$ case. 

The semigroup algebra $\C[B_1]$ is generated by two elements of level $1$ and an element of level $2,$ and no relations hold between these elements.  The generators of the semigroup algebra $\C[B_2]$ are given in Figure \ref{Fig5}; a Markov basis for the toric ideal vanishing on these generators can be computed using the software package 4ti2 \cite{4ti2}. This computation shows that the relations for $\C[B_2]$ are generated by those of of levels $2$, $3$, and $4$, we omit the details.

\subsection{Relations for $\C[P_{\Gamma_{g, 1}}]$}

We describe relations for $\C[P_{\Gamma_{g, 1}}]$ by building relations for this semigroup out those of $B_1$ and $B_2.$  The theorem holds for $B_1$ because this semigroup has unique factorization, this is the base case of our induction.  To treat $P_{\Gamma_{g, 1}}$, we consider it as a toric fiber product of $P_{\Gamma_{g-1, 1}}$ with $B_2$ over the semigroup of non-negative integer points $(n, L)$ with $2n \leq L$. This semigroup
is generated by $(0, 1)$ and $(1, 2)$, and also has unique factorization.  The maps from $P_{\Gamma_{g-1, 1}}$ and $B_2$ to this semigroup are computed by sending a lattice point to half the weight on the edge shared by the supporting graphs of these semigroups, and the level of the weightings. Given an element $\omega \in P_{\Gamma_{g, 1}}$, and two factorizations we first collect level $1$ terms into level $2$ terms; this gives factorizations into only level $2$ elements in the even level case, and level $2$ elements except one level $1$ element in the odd case.  Note that in doing so, we have used at most level $2$ relations.  

Now we consider the restriction of this relation to $P_{\Gamma_{g-1, 1}}$.  By induction, 
there is a way to transform $\sum \beta_i^*|_{g-1, 1}$ into $ \sum \alpha_j^*|_{g-1, 1}$
using degree $2, 3, 4$ relations.  We claim that each of these relations can be lifted
to a relation in $P_{\Gamma_{g, 1}}.$ Any such relation must preserve the list of values
along the shared edge, this means that we can ``unglue'' the $B_2$ side of the weightings involved
in the relation, perform the relation on the $(g-1, 1)$ side, and glue the $B_2$ weightings back on in 
any way consistent with their values along the shared edge to obtain a relation on $P_{\Gamma_{g,1}},$ see Figure \ref{Fig8}.

\begin{figure}[htbp]
\centering
\includegraphics[scale = 0.5]{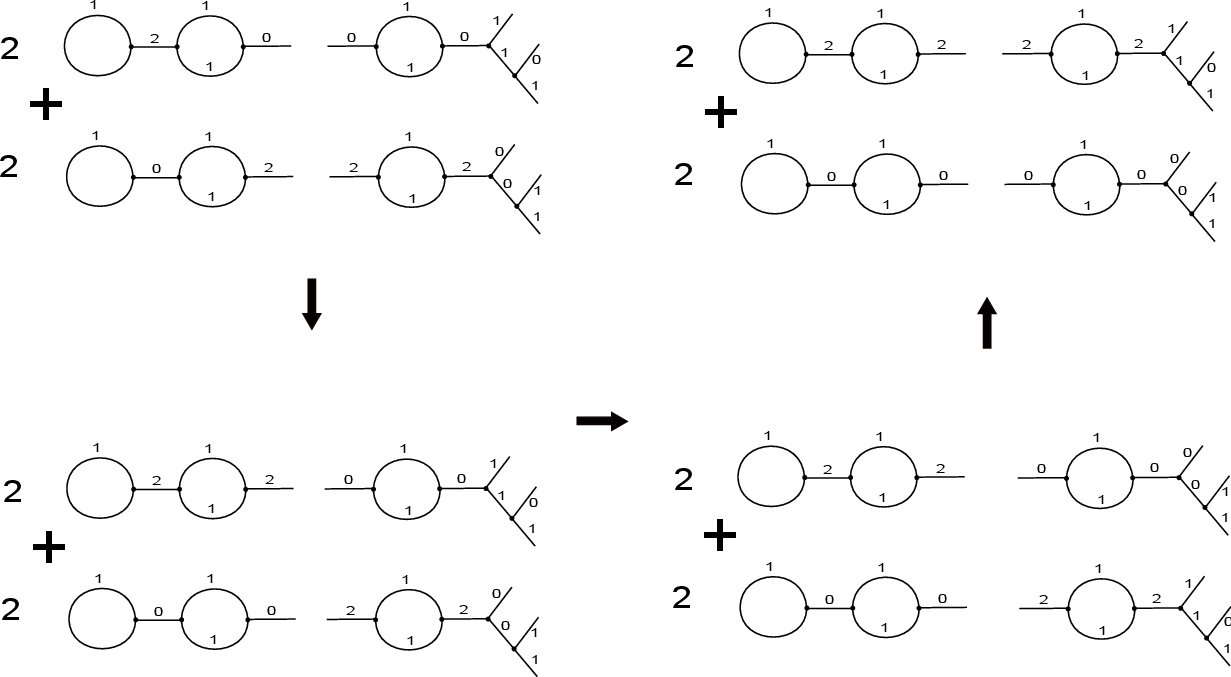}
\caption{}
\label{Fig8}
\end{figure}

This proves that we can transform $\sum \beta_i^*$ to  $\sum \alpha_j^*$ ``along $P_{\Gamma_{g-1, 1}}$.''
But we can now apply the same argument to perform this transformation over $B_2$, using the
relations above.  This completes the proof of the relation statement in Theorem \ref{sl2}
for $(g, 1).$

\subsection{Relations for $\C[P_{\Gamma_{g, n}}]$}

The same argument used in the case $P_{\Gamma_{g, 1}}$ to extend relations from $P_{\Gamma_{g-1, 1}}$ to $P_{\Gamma_{g, 1}}$ can 
be used to show that any relation on $P_{\Gamma_{g, 1}}$ can be extended to a relation on $P_{\Gamma_{g, n}}$.  From this
it follows that given two normalized ways to represent an element $\omega = \sum \beta_i^* = \sum \alpha_j^*,$
one can be transformed to the other ``over $P_{\Gamma_{g, 1}}$'' with relations of degree $2, 3, 4.$  Now the same argument
can be applied over the edge which separates $P_{\Gamma_{0, n+1}}$ from $P_{\Gamma_{g, 1}}$ using the degree $2$  $P_{\Gamma_{0, n+1}}$ relations discovered in \cite{BW}.  Our analysis takes care of Theorem \ref{degensl2} for $n > 0$, the $n=0$ case is handled by noting that $\C[P_{\Gamma_{g, 0}}] \subset \C[P_{\Gamma_{g, 1}}]$ is spanned by the set of elements with $0$ weighting the leaf, and all of our techniques specialize to this case without alteration.

\subsection{Best possible generation of $\C[P_{\Gamma}]$}

We sketch how our quadratic generation result is in a certain sense best possible for the semigroup algebras $\C[P_{\Gamma}]$ when $g > 0$. 
We say a trivalent graph $\Gamma'$ is a trivalent minor of a trivalent graph $\Gamma$ if $\Gamma'$
can be obtained from a subgraph of $\Gamma$ by replacing pairs of edges joined at a 
bivalent vertex with a single edge, and any leaf of $\Gamma'$ is also a leaf of $\Gamma.$ 

\begin{proposition}
If $\Gamma$ contains a trivalent minor isomorphic to one of the graphs depicted in Figure \ref{counter}, 
then $\C[P_{\Gamma}]$ has indecomposable elements of degree $\geq 2.$
\end{proposition}

\begin{figure}[htbp]
\centering
\includegraphics[scale = 0.4]{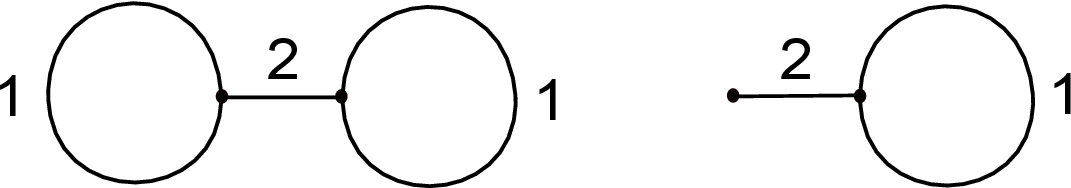}
\caption{Trivalent minors with indecomposable weightings}
\label{counter}
\end{figure}

This proposition holds for the graphs depicted in Figure \ref{counter}, and any indecomposable element on a trivalent minor of a graph $\Gamma$ gives an indecomposable on $\Gamma$ by extending this element by $0$ to the rest of $\Gamma$.   Notice that any trivalent graph with $g, n \geq 1$ contains the graph on the right in Figure \ref{counter} as a trivalent minor, for this reason we focus on $\Gamma$ with $n = 0.$  For $g = 2, 3, 4,$
it can be shown that the graphs depicted in Figure \ref{counter2} do not contain either of the graphs in Figure \ref{counter}, and that these are the only trivalent graphs with this property.    Any trivalent graph $\Gamma$ with $g \geq 5$ which does not have either  graph in Figure \ref{counter} as a trivalent minor must have the graph on the right in Figure \ref{counter2} as a $g = 4$ minor.  Any extension of this graph to a trivalent $g = 5$ graph can then be checked to have the graph on the left hand side of Figure \ref{counter} as a trivalent minor.  This shows our results are best possible, excluding the cases $n =0, g = 2, 3, 4$.

\begin{figure}[htbp]
\centering
\includegraphics[scale = 0.5]{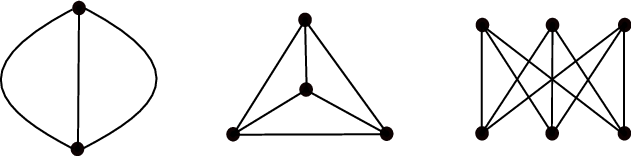}
\caption{}
\label{counter2}
\end{figure}

\section{Acknowledgements}
We thank  Weronika Buczy\'{n}ska, Edward Frenkel, Noah Giansiracusa, Kaie Kubjas, Shrawan Kumar, Eduard Looijenga, Diane Maclagan, John Millson, Steven Sam, Bernd Sturmfels, David Swinarski, and Fillipo Viviani for useful discussions. We also thank the reviewer for many helpful suggestions.  This paper was mostly written at the fall 2009 Introductory Workshop in Tropical Geometry at MSRI.

\bibliographystyle{alpha}
\bibliography{Biblio}

\bigskip
\noindent
Christopher Manon:\\
Department of Mathematics,\\ 
George Mason University,\\ 
Fairfax, VA 22030 USA

\end{document}